\documentclass[11pt]{article}
\usepackage[margin=1.1in]{geometry}
\usepackage{amsmath,amssymb,amsthm}
\usepackage[T1]{fontenc}
\usepackage{lmodern}
\usepackage{booktabs,array}
\usepackage{microtype}
\newtheorem{theorem}{Theorem}
\newtheorem{lemma}{Lemma}
\theoremstyle{remark}
\newtheorem{remark}{Remark}
\newtheorem{proposition}[theorem]{Proposition}
\newtheorem{corollary}[theorem]{Corollary}
\theoremstyle{definition}
\newtheorem{definition}[theorem]{Definition}
\newcommand{\Z}{\mathbb{Z}}
\newcommand{\N}{\mathbb{N}}
\newcommand{\Fp}{\mathbb{F}_p}
\newcommand{\Fq}{\mathbb{F}_q}
\newcommand{\md}{\ \mathrm{mod}\ }
\newcommand{\tsub}{\mathbin{\dot{-}}}
\newcommand{\HW}{\operatorname{HW}}
\newcommand{\F}{\mathbb{F}}
\newcommand{\Zp}{\Z_p}
\newcommand{\vp}{v_p}
\newcommand{\legendre}[1]{\genfrac{(}{)}{0.4pt}{1}{#1}{p}}

\DeclareMathOperator{\GF}{GF}
\title{Short arithmetic terms for the cardinality of elliptic curves over rings of remainder classes and over finite fields}
\author{Bogdan Dumitru\footnote{ Faculty of Mathematics and Computer Science, University of Bucharest, Academiei 14, Bucharest (RO-010014), Romania.}, Mihai Prunescu \footnote{Research Center for Logic, Optimization and Security (LOS), Faculty of Mathematics and Computer Science, University of Bucharest, Academiei 14, Bucharest (RO-010014), Romania.} \footnote{ Simion Stoilow Institute of Mathematics of the Romanian Academy, Research unit 5, P. O. Box 1-764, Bucharest (RO-014700), Romania. E-mail: {\tt mihai.prunescu@imar.ro}}
}

\begin{document}
\date{}
\maketitle

\begin{abstract}\noindent

Arithmetic terms are fixed finite compositions of addition, truncated
subtraction, multiplication, integer division and exponentiation on natural
numbers. We construct such terms for the number of affine solutions of
$y^2=x^3+Ax+B$. For arbitrary moduli $n\ge1$, a specialization of
Prunescu's general construction reduces the number of monomial
contributions from $49$ to $15$ and the size of the packed integer from
approximately $2n^{11}$ to approximately $2n^5$ binary digits. For prime
moduli $p\ge17$, the Hasse invariant and the trace of Frobenius give a
term of about thirty operations. For curves defined over $\Fp$, a further
term counts the points over $\mathbb F_{p^k}$ with $k$ as a variable.

For primes $p\ge5$, we also obtain arithmetic terms for the counts over
$\Z/p^k\Z$ with $k$ variable, covering good reduction, nodes and cusps,
including nonminimal equations and zero discriminant.

\end{abstract}

\noindent

\section{Introduction}\label{sec:introduction}

We study the number $N(A,B,n)$ of pairs $(x,y)\in(\Z/n\Z)^2$ satisfying
$y^2=x^3+Ax+B$. The coefficients $A,B$ are natural numbers unless a wider
domain is stated. Any integer coefficients can be replaced by their
least nonnegative residues modulo the modulus being considered.

An \emph{arithmetic term} is a fixed finite composition of addition,
truncated subtraction $x\tsub y=\max(x-y,0)$, multiplication, integer
division $\lfloor x/y\rfloor$ and exponentiation on natural numbers.
Division is used with positive denominators. The remainder
$x\md y=x\tsub y\lfloor x/y\rfloor$ is therefore an arithmetic term.
Throughout, we use the convention $0^0=1$.

Mazzanti proved that every Kalmar elementary function is expressible by
an arithmetic term \cite{Mazzanti2002plainbases}. Prunescu applied this
method to count the solutions of polynomial systems over $\Z/n\Z$
\cite{Pru26}. For the Weierstra\ss{} equation $y^2=x^3+Ax+B$, the resulting
term has the form

$$N(A,B,n)=\frac{\HW(M(A,B,n))}{n+90}-(n+1)^{10}.$$

Here $\HW$ counts the ones in the binary representation of an integer.
The construction has $49$ monomial contributions, expressed through the
generalized geometric progressions
$G_r(q,t)=\sum_{j=0}^t j^r q^j$, with indices up to $12$.
Further arithmetic-term representations can be found in
\cite{istrateprunescushuniauniversal,
prunescu2024numbertheoreticfunctions,prunescushuniaprimes}.

For this curve family, the five-variable construction can be replaced
by a three-variable construction in which each variable has its own
range. Theorem~\ref{thm:composite} gives the resulting term for all
$n\ge1$. It has $15$ monomial contributions and uses $G_0,\ldots,G_6$.
The packed integer has approximately $2n^5$ binary digits, compared with
approximately $2n^{11}$ in the general construction.

For a nonsingular cubic and a prime modulus $p\ge17$, a coefficient of
$(x^3+\bar A x+\bar B)^{(p-1)/2}$ determines the trace of Frobenius modulo
$p$, where $\bar A=A\md p$ and $\bar B=B\md p$. We extract this
coefficient as a digit of an integer and recover the trace using Hasse's
bound. Put

$$r \;=\; \left (\left \lfloor  \frac{ \left (2^{3p^2} + \bar A\, 2^{p^2} + \bar B\right )^{\lfloor p/2\rfloor}}{
 2^{p^2(p-1)}} \right\rfloor \md 2^{p^2}\right ) \md p .$$

Then Theorem~\ref{thm:ztopz} gives

$$N(A,B,p) \;=\; (p \tsub r) \;+\; p\cdot\Bigl(1 \tsub \bigl((4p+1)\tsub r^2\bigr)\Bigr).$$

The same formula covers singular cubics, for which $p-N(A,B,p)$ has
absolute value at most $1$. For nonsingular cubics defined over $\Fp$,
the trace also determines the counts over $\mathbb F_{p^k}$ through a
linear recurrence. Theorem~\ref{thm:term} expresses these counts by one
arithmetic term with $k$ variable; Theorem~\ref{thm:ext}(ii) gives the
corresponding formula for singular cubics.

For the rings $\Z/p^k\Z$, the count depends on the reduction of the
given integral equation. When $p\ge5$, Hensel's lemma gives
$N(A,B,p^k)=p^{k-1}N(A,B,p)$ in the case of good reduction. For a node,
the count is determined by $k$, $v_p(4A^3+27B^2)$ and one Legendre symbol.
For a cusp, successive substitutions give the finite table of
Theorem~\ref{thm:cusp}. Theorem~\ref{thm:compression} sums all repeated
scaling steps, and Proposition~\ref{prop:R} supplies the remaining
cubic root count. These local formulas include nonminimal equations and
zero discriminant. We restrict the local analysis to $p\ge5$; the
all-moduli term also applies to powers of $2$ and $3$.

The prime-field term uses about thirty arithmetic operations. Each
operation is counted as one step, regardless of operand size. The terms
for extension fields and prime-power rings use further operations,
whose number is independent of $k$. Sizes of the intermediate integers and numerical
evaluations are discussed in Remark~\ref{rem:width} and
Section~\ref{sec:numerics}.

\section{An arithmetic term for all moduli}\label{sec:composite}

For arbitrary $n\ge1$, we count the solutions of the congruence by
encoding a bounded polynomial in one integer. The construction uses
the binary weight $\HW$ and the generalized geometric progressions
$G_r$. We recall the arithmetic terms for $\HW$ and its ingredients
\cite{Pru26,prunescu2024numbertheoreticfunctions}:

\begin{align*}
\gcd(a,b) &= \left ( \left \lfloor{\frac{5^{ab(ab + a+b)}}{(5^{a^2b}-1)(5^{b^2a}-1)}} \right \rfloor \bmod 5^{ab} \right ) \dot{-} 1, \\
\nu_2(n) &= \left \lfloor{\frac{\gcd(n, 2^n)^{n+1} \bmod (2^{n+1}-1)^2}{2^{n+1}-1}} \right \rfloor , \\
\binom{n}{k} &= \left \lfloor{\frac{(2^n+1)^n}{2^{nk}}} \right \rfloor \bmod 2^n , \\
\HW(n) &= \nu_2 \left( \binom{2n}{n} \right).
\end{align*} 

Also, recall the generalized geometric progressions for $q>1$ and $t\ge0$
\[
G_r(q,t)=\sum_{j=0}^{t}j^r q^j .
\]
Their closed forms are obtained from
\[
G_0(q,t)=\frac{q^{t+1}-1}{q-1},
\]
\begin{align*}
 G_{r + 1} ( q , t ) = \frac{\partial}{\partial q} G_r ( q , t + 1 ) - \sum_{j=0}^r \binom{r+1}{j} G_j ( q , t ) .
\end{align*}
with $r$ fixed. Repeated differentiation gives a quotient with denominator
$(q-1)^{r+1}$ and a fixed integer-polynomial numerator. Splitting that
numerator into its positive and negative monomials and replacing their
difference by truncated subtraction gives an arithmetic term. Only
$G_0,\dots,G_6$ occur below; here $q-1$ may of course be written
$q\tsub1$. We note for example that:
\[
G_1(q,t)=\frac{tq^{t+2}-(t+1)q^{t+1}+q}{(q-1)^2}.
\]

\begin{theorem}\label{thm:composite}
Let $n\ge 1$ and $A,B\in\N$. Put $\bar A=A\md n$, $\bar B=B\md n$, $m=n^3\tsub\bar B$,
\[
w=n+12,\qquad T=n^2+n+1,\qquad q=2^{2w},
\]
and for $i,j,k\ge 0$ write
\[
S(i,j,k) \;=\; G_i(q,\,n-1)\; G_j(q^{\,n},\,n-1)\; G_k(q^{\,n^2},\,T-1).
\]
Define the two natural numbers
\begin{align*}
S_E^+ \;=\;& \phantom{{}+{}} S(0,4,0) \;+\; m^2\,S(0,0,0) \;+\; S(6,0,0)
        \;+\; {\bar A}^2\,S(2,0,0) \;+\; n^2\,S(0,0,2)\\
        &{}+ 2m\,S(0,2,0) \;+\; 2{\bar A}\,S(4,0,0) \;+\; 2n\,S(3,0,1) \;+\; 2{\bar A}n\,S(1,0,1)\\
S_E^- \;=\;& \phantom{{}+{}} 2\,S(3,2,0) \;+\; 2{\bar A}\,S(1,2,0) \;+\; 2n\,S(0,2,1)\\
         &{}+ 2m\,S(3,0,0) \;+\; 2{\bar A}m\,S(1,0,0) \;+\; 2mn\,S(0,0,1),
\end{align*}
and
\[
M \;=\; \bigl(2^w\tsub1\bigr)
\Bigl(\bigl(2^w+1\bigr)S(0,0,0)+S_E^-\;\tsub\;S_E^+\Bigr).
\]
Then the number $N$ of points $(x,y)\in(\Z/n\Z)^2$ with $y^2=x^3+Ax+B$ equals
\[
N \;=\; \frac{\HW(M)}{w} \;\tsub\; n^2T .
\]
\end{theorem}

The proof consists of three lemmas. Fix $n$, $\bar A$, $\bar B$, $m$, $w$, $T$ and
$q$ as in the theorem. Set $g(x)=x^3+\bar Ax+\bar B$ and
\[
P(x,y,z) \;=\; y^2+m-x^3-\bar Ax-nz, \qquad E \;=\; P^2 ,
\]
on the box $\mathcal B=[0,n-1]\times[0,n-1]\times[0,T-1]$. Lemma
\ref{lem:box} shows that $N$ equals the number of zeros of $E$ in
$\mathcal B$, and that $E$ never exceeds $2^w$ on $\mathcal B$. Lemma
\ref{lem:delta} attaches to every value of $E$ a number whose binary weight
distinguishes the zero cells from the nonzero cells. Lemma \ref{lem:pack}
shows that $M$ collects these numbers over all cells, so that $\HW(M)$
counts the zeros.

\begin{lemma}[zeros of $E$ in the box]\label{lem:box}
The number of triples $(x,y,z)\in\mathcal B$ with $E(x,y,z)=0$ equals $N$.
Moreover $0\le E(x,y,z)\le 2^w$ for every $(x,y,z)\in\mathcal B$.
\end{lemma}

\begin{proof}
Since $0\le\bar B<n$, the truncated subtraction defining $m$ is exact and
$m=n^3-\bar B$. For $x,y\in[0,n-1]$ put $D=y^2+m-x^3-\bar Ax$. From $0\le x^3+\bar Ax\le n^3-2n^2+n$
and $n^3-n+1\le m\le n^3$ we obtain $1\le D\le n^3+n^2$. Since $m\equiv
-\bar B\pmod n$, we have $D\equiv y^2-g(x)\pmod n$. Hence the congruence
$y^2\equiv g(x)\pmod n$ holds if and only if $n$ divides $D$. Because $1\le
D\le n^3+n^2$, this happens if and only if $D=nz$ for an integer $z$ with
$1\le z\le n^2+n$, and this $z$ is unique and lies in $[0,T-1]$. Each of the
$N$ solutions $(x,y)$ of the congruence therefore yields exactly one zero of
$P$, and hence of $E$, in $\mathcal B$, and every zero arises in this way.
For the bound, $D\in[1,\,n^3+n^2]$ and $nz\in[0,\,n(T-1)]=[0,\,n^3+n^2]$
imply $|P|\le n^3+n^2$, so $E\le (n^3+n^2)^2$. The inequality
$(n^3+n^2)^2\le 2^{n+12}$ holds for every $n\ge1$. The smallest relative margin occurs
at $n=8$, where $(n^3+n^2)^2=331776$ and $2^{20}=1048576$.
\end{proof}

\begin{lemma}[binary weight of a cell value]\label{lem:delta}
For $0\le E\le 2^w$ let $\delta(E)=(2^w-1)(2^w-E+1)$. Then
$0\le\delta(E)\le 2^{2w}-1$, and
\[
\HW(\delta(E)) \;=\;
\begin{cases}
2w, & E=0,\\
w, & 1\le E\le 2^w.
\end{cases}
\]
\end{lemma}

\begin{proof}
The value $\delta(0)=2^{2w}-1$ has binary weight $2w$. For $1\le E\le 2^w$
put $u=2^w-E+1\in[1,2^w]$ and $v=u-1\in[0,2^w-1]$. Then
\[
\delta \;=\; (2^w-1)\,u \;=\; u\,2^w-u \;=\; (u-1)\,2^w+(2^w-u)
\;=\; v\,2^w+\bigl((2^w-1)-v\bigr).
\]
The two summands occupy disjoint binary positions. The number $(2^w-1)-v$ is
the bitwise complement of $v$ in $w$ bits, so
$\HW(\delta)=\HW(v)+\bigl(w-\HW(v)\bigr)=w$.
\end{proof}

\begin{lemma}[closed form of the packing]\label{lem:pack}
The map $v(x,y,z)=x+ny+n^2z$ is a bijection from $\mathcal B$ onto
$[0,\,n^2T-1]$, and
\[
M \;=\; \sum_{(x,y,z)\in\mathcal B} q^{\,v(x,y,z)}\;\delta\bigl(E(x,y,z)\bigr).
\]
\end{lemma}

\begin{proof}
The polynomial $P$ has the five terms $y^2$, $m$, $-x^3$, $-\bar A x$ and $-nz$, so
its square $E$ expands into $15$ monomials in $x$, $y$ and $z$: the five
squares $y^4$, $m^2$, $x^6$, ${\bar A}^2x^2$ and $n^2z^2$, and ten mixed products.
Because $v$ is linear in $(x,y,z)$ and the three ranges are independent, the
sum of each monomial over $\mathcal B$ factors into three generalized
geometric progressions:
\[
\sum_{\mathcal B} x^i y^j z^k\, q^{\,x+ny+n^2z}
 \;=\; G_i(q,\,n-1)\,G_j(q^n,\,n-1)\,G_k(q^{n^2},\,T-1) \;=\; S(i,j,k).
\]
Collecting the $15$ monomials with their coefficients yields
\[
S_E^+-S_E^-=\sum_{\mathcal B} E\,q^{v},
\]
and the case $i=j=k=0$ yields $S(0,0,0)=\sum_{\mathcal B}q^v$. Consequently
\[
\bigl(2^w+1\bigr)S(0,0,0)+S_E^- - S_E^+
=\sum_{\mathcal B}q^v\bigl(2^w+1-E\bigr)\geq 0,
\]
so the truncated subtraction in the definition of $M$ is exact. Therefore
\[
M=\bigl(2^w-1\bigr)\Bigl(\bigl(2^w+1\bigr)S(0,0,0)+S_E^- -S_E^+\Bigr)
=\sum_{\mathcal B} q^{v}\,\bigl(2^w-1\bigr)\bigl(2^w+1-E\bigr)
=\sum_{\mathcal B} q^{v}\,\delta(E). \qedhere
\]
\end{proof}

\begin{proof}[Proof of Theorem \ref{thm:composite}]
By Lemmas \ref{lem:box} and \ref{lem:delta}, every coefficient $\delta(E)$ in
the sum of Lemma \ref{lem:pack} lies in $[0,\,q-1]$. The exponents $v$ are
pairwise distinct, so this sum is the base-$q$ representation of $M$. The
binary digits of distinct cells therefore occupy disjoint positions, and the
binary weights of the cells add. By Lemma \ref{lem:box} exactly $N$ of the
$n^2T$ cells have $E=0$, so Lemma \ref{lem:delta} gives
\[
\HW(M) \;=\; 2w\cdot N \;+\; w\cdot\bigl(n^2T-N\bigr) \;=\; w\,\bigl(N+n^2T\bigr).
\]
Thus $\HW(M)/w=N+n^2T$, and the truncated subtraction in the formula for
$N$ is exact. All displayed constituents of $M$ are arithmetic terms over
the natural numbers, so the displayed formula is itself an arithmetic term.
\end{proof}

\begin{remark}[comparison with the general construction]
The general construction of \cite{Pru26}, applied to this curve family at
modulus $n$, produces a term with five variables, a box of $(n+1)^{10}$
cells, digit width $n+90$, $49$ monomial contributions, and progressions up
to $G_{12}$. The term of Theorem~\ref{thm:composite} has three variables, a
box of $n^2T$ cells, digit width $n+12$, $15$ monomial contributions, and
progressions up to $G_6$. The integer $M$ of Theorem~\ref{thm:composite} has
$2(n+12)\,n^2\,T$ binary digits, which is approximately $2n^5$. The
corresponding integer of the general term has approximately
$2(n+90)(n+1)^{10}$ binary digits, which is approximately $2n^{11}$. At $n=6$
the two sizes are $55{,}728$ bits and approximately $5.4\cdot 10^{10}$ bits.
The ratio of the two sizes grows proportionally to $n^6$. The reduction comes
from giving each variable its own range instead of one common range.
\end{remark}

\section{Facts}\label{sec:facts}

This section collects the definitions and the standard results used in
the rest of the paper. Subsection \ref{sec:facts-fp} concerns the prime
field $\Fp$ and is used in Sections \ref{sec:tracefp}--\ref{sec:formulafpk};
Subsection \ref{sec:facts-zpk} concerns the rings $\Z/p^k\Z$ with $k\ge2$
and is used from Section \ref{sec:zoverptokz} on.

\subsection{Finite fields}\label{sec:facts-fp}

Let $p$ be an odd prime. Denote by $N$ the number of points $(x,y) \in \Fp^2$ satisfying $y^2 = x^3 + Ax + B$, where the parameters $A, B \in \N$. Put $\bar A = A \bmod p$ and $\bar B = B \bmod p$.

Write $f(x)=x^3+\bar A x+\bar B\in\Z[x]$. Let $\chi$ denote the symbol of Legendre modulo $p$:
$$\chi(u) = \left ( \frac{u}{p} \right ) =
\begin{cases}
    1, & \textrm{if } \exists v \neq 0\,\, u = v^2, \\
    0, & \textrm{if } u = 0, \\
    -1, & \textrm{otherwise.}
\end{cases}$$
Observe that $\chi$ is a character of the group $(\Fp \setminus \{0\}, \cdot)$, extended with $\chi(0) = 0$.

We shall use the following elementary facts about $\Fp$.

\begin{lemma}[Fermat, Euler]\label{lem:eulerfermat}
Let $p$ be an odd prime and $u\in\Fp$.
\begin{itemize}
\item[(i)] If $u\ne0$ then $u^{p-1}=1$. Consequently $1-u^{p-1}$ is the
indicator of $u=0$.
\item[(ii)] (Euler's criterion) $u^{(p-1)/2}=\chi(u)$ in $\Fp$.
\end{itemize}
\end{lemma}

\begin{proof}
(i) The group $\Fp^{\times}$ has order $p-1$, so $u^{p-1}=1$ by Lagrange's
theorem. (ii) For $u=0$ both sides vanish. For $u\ne0$,
$(u^{(p-1)/2})^2=u^{p-1}=1$, so $u^{(p-1)/2}=\pm1$. The polynomial
$X^{(p-1)/2}-1$ has at most $(p-1)/2$ roots in the field $\Fp$. The map
$v\mapsto v^2$ is two-to-one on $\Fp^{\times}$, so there are exactly
$(p-1)/2$ nonzero squares, and each square $v^2$ satisfies
$(v^2)^{(p-1)/2}=v^{p-1}=1$. Hence the roots of $X^{(p-1)/2}-1$ are exactly
the nonzero squares, and the non-squares give $-1$.
\end{proof}

The number of solutions $y\in\Fp$
of $y^2=u$ is $1+\chi(u)$ for every $u\in\Fp$. Summing over $x$ gives
\begin{equation}\label{eq:count}
N \;=\; \sum_{x\in\Fp}\bigl(1+\chi(f(x))\bigr) \;=\; p+\sum_{x\in\Fp}\chi(f(x)).
\end{equation}

Denote by $a_p$ the quantity:
$$a_p = -\sum_{x\in\Fp}\chi(f(x)) = p - N.$$

If the cubic is nonsingular, $a_p$ is the {\bf trace of Frobenius} of the
resulting elliptic curve; see \cite{Sil09}. In the singular case we retain
the notation $a_p$ for the character-sum quantity $p-N$.

Abbreviate
$e=\lfloor p/2\rfloor=\tfrac{p-1}{2}$ and $w=p^2$. With these abbreviations, consider the expression:
\begin{equation}\label{eq:r}
r=\left ( \left \lfloor \frac{(2^{3w}+\bar A 2^w+\bar B)^e}{2^{w(p-1)}} \right \rfloor \md 2^w \right )\md p.
\end{equation}

Suppose now that the cubic is nonsingular, that is, $4\bar A^3+27\bar B^2\not\equiv 0\pmod p$, so that the equation defines an elliptic curve $E$ over $\Fp$. Its projective closure has exactly one point at infinity, hence $\# E(\Fp) = N+1$. By Hasse's Theorem, \cite[Theorem V.1.1]{Sil09}, $|\# E(\Fp) - (p+1)| \leq 2\sqrt p$, that is,
$$| N - p | \leq 2 \sqrt{p}. $$
For $n \geq 1$ denote by $N_n$ the number of points $(x,y) \in \mathbb F_{p^n}^2$ with $y^2 = f(x)$, so that $N_1 = N$ and $\# E(\mathbb F_{p^n}) = N_n + 1$. According to \cite[Theorem V.2.3.1]{Sil09}, if $\alpha, \beta \in \mathbb C$ are the solutions of
$$T^2 - a_p T + p = 0,$$
then $|\alpha| = |\beta| = \sqrt{p}$ and $\# E(\mathbb F_{p^n}) = p^n + 1 - \alpha^n - \beta^n$ for every $n \geq 1$; in affine terms,
$$ N_n = p^n - \alpha^n-\beta^n. $$
It follows that the formal series:
$$\sum_{k=1}^{\infty} N_k T^k = \sum_{k=1}^{\infty} (p T)^k - \sum_{k=1}^{\infty} (\alpha T)^k - \sum_{k=1}^{\infty} (\beta T)^k = $$
$$= \frac{pT}{1-pT} - \frac{\alpha T}{1-\alpha T} - \frac{\beta T}{1-\beta T}.$$
A standard computation shows that this series is a rational function with coefficients in $\mathbb Z$.

In \cite{PS24}, the following term-extraction result is proved:

\begin{theorem}\label{ThmExtraction1}

Let $t(n)$ be a sequence of natural numbers, and let $R$ be the radius
of convergence at zero of its generating function
\[
\GF_t(z)=\sum_{n=0}^{\infty}t(n)z^n.
\]
Suppose that the integers $b,m,n$ satisfy $b\ge2$, $n\ge m\ge2$,
$b^{-m}<R$, and $t(r)<b^{r-2}$ for all integers $r\ge m$. Then

\begin{equation}\label{IdentityExtraction}
t(n)=\left\lfloor b^{n^2}\GF_t(b^{-n})\right\rfloor\bmod b^n.
\end{equation}
\end{theorem}

A sequence satisfying a linear recurrence with constant coefficients
is called C-recursive; its generating function is rational. If the
sequence is nonnegative and the hypotheses of
Theorem~\ref{ThmExtraction1} hold, substituting this rational function
gives an arithmetic term for its $n$-th entry. For an integer
C-recursive sequence $s(n)$ that takes negative values, one first adds
$c^{n+1}$ with $c$ large enough to obtain a nonnegative sequence
\cite{PS24}. Theorem~\ref{thm:term} uses this construction for the
counts over extension fields.

The following two lemmas are simple technical observations. The first is used in
Lemmas \ref{lem:cong} and \ref{lem:bound} and in Proposition \ref{prop:R},
the second in Lemma \ref{lem:extract} and in Proposition \ref{prop:R}.

\begin{lemma}\label{lem:sumofpowers} Let $p$ be an odd prime. Then:
    $$S_k = \sum_{x \in \Fp} x^k = \begin{cases}
        -1, & \textrm{if } k > 0 \wedge (p-1) | k ,\\
        0, & \textrm{otherwise.}
    \end{cases}$$
\end{lemma}

\begin{proof}
    With the convention $0^0=1$, $S_0 = p = 0$ in $\Fp$. $S_{m(p-1)} = p-1 = -1$. For all other $k \in \mathbb N$,  there is some $c \in \Fp \setminus \{0\}$ such that $c^k \neq 1$. As $x \leadsto cx$ is a permutation of $\Fp$,
    $$S_k(c^k - 1) = 0,$$
    so $S_k = 0$.
\end{proof}

\begin{lemma}\label{lem:inequality}
    If $n \geq 3$ is a natural number, then
    $$(2n-1)^{n-1} < 2^{n^2}.$$
\end{lemma}

\begin{proof}
    $$(2n-1)^{n-1} < 2^n n^n = 2^{n(1+\log_2 n)} < 2^{n^2},$$
    because for such numbers $1 + \log_2 n < n$.
\end{proof}

\subsection{The rings $\Z/p^k\Z$}\label{sec:facts-zpk}

{\bf $p$-adic notation.} Throughout this subsection and in Sections
\ref{sec:zoverptokz}--\ref{sec:tate}, $p\ge5$ is a prime. We write $\Zp$
for the ring of $p$-adic integers, $\vp$ for the $p$-adic valuation with
$\vp(0)=\infty$, and $\Zp^{\times}=\{u\in\Zp:\vp(u)=0\}$ for the group of
units. Every nonzero $x\in\Zp$ is uniquely $p^{\vp(x)}$ times a unit, and
$\vp(xy)=\vp(x)+\vp(y)$, $\vp(x+y)\ge\min(\vp(x),\vp(y))$ with equality
when $\vp(x)\ne\vp(y)$. For $k\ge1$ the rings $\Z/p^k\Z$ and
$\Zp/p^k\Zp$ are the same, so ``modulo $p^k$'' makes sense for $p$-adic
integers, and a polynomial with coefficients in $\Zp$ defines a function
on $(\Z/p^k\Z)^2$. For $a\in\Z$ and $b\in\Z$ with $p\nmid b$, the fraction
$a/b$ denotes the corresponding element of $\Zp$. For a unit $u$ we put
$\chi(u)=\chi(u\bmod p)$; then $\chi(a/b)=\chi(ab)$, because
$a/b=ab\cdot b^{-2}$ and $b^{-2}$ is a nonzero square modulo $p$. Finally,
for $0\le a\le k$, the residues $x$ modulo $p^k$ with $p^a\mid x$ are
exactly $x=p^aX$ with $X$ modulo $p^{k-a}$; there are $p^{k-a}$ of them.

{\bf The hypothesis $p\ge5$.} Fix $A,B\in\Zp$, put $f(x)=x^3+Ax+B$ and
$\Delta=4A^3+27B^2$. The discriminant of the Weierstra\ss{} equation
$y^2=f(x)$ is $-16\Delta$ \cite[III.1]{Sil09}, so for $p\ge5$ the
valuation $\vp(\Delta)$ is the valuation of the discriminant. The
hypothesis $p\ge5$ is used in three ways: $2$ and $3$ are units, so the
substitution $x\mapsto x+\tilde\rho$ of Lemma \ref{lem:shift} and the
coefficient comparisons of Lemma \ref{lem:cubic} are available; the
dichotomy node/cusp of Lemma \ref{lem:cubic}(iii) needs $p\ne3$, since
over $\F_3$ the polynomial $x^3-1=(x-1)^3$ has a triple root although its
constant term is a unit; and the lifting Lemma \ref{lem:hensel} uses the
partial derivative $2y$, which vanishes identically for $p=2$.

\begin{lemma}[cubics over $\Fp$]\label{lem:cubic}
Let $p\ge5$, $a,b\in\Fp$, $g(x)=x^3+ax+b$ and $D=4a^3+27b^2$.
\begin{itemize}
\item[(i)] $g$ has a repeated root in $\overline{\Fp}$ if and only if $D=0$.
\item[(ii)] If $D=0$, the repeated root $\rho$ is unique, it lies in
$\Fp$, and $g=(x-\rho)^2(x+2\rho)$; in particular
\[
  a=-3\rho^2,\qquad b=2\rho^3 .
\]
\item[(iii)] If $D=0$, then $\rho$ is a triple root if and only if
$\rho=0$, if and only if $a=b=0$. If $a\ne0$, then $\rho=-3b/(2a)\ne0$ is
a double root; if $a=0$, then $b=0$.
\item[(iv)] If $D=0$, the number of $(x,y)\in\Fp^2$ with $y^2=g(x)$ is
$p-\chi(3\rho)$; it equals $p$ for a triple root, and for a double root
$\chi(3\rho)=\chi(-2ab)\in\{\pm1\}$.
\end{itemize}
\end{lemma}

\begin{proof}
(i) An element $\rho\in\overline{\Fp}$ is a repeated root of $g$ if and
only if $g(\rho)=g'(\rho)=0$. If so, $g'(\rho)=3\rho^2+a=0$ gives
$a=-3\rho^2$, and then $g(\rho)=\rho^3-3\rho^3+b=0$ gives $b=2\rho^3$;
hence $D=4\cdot(-27\rho^6)+27\cdot4\rho^6=0$. Conversely, let $D=0$. If
$a=0$ then $b=0$ and $0$ is a triple root. If $a\ne0$, put
$\rho=-3b/(2a)$. From $27b^2=-4a^3$ we get
$\rho^2=9b^2/(4a^2)=-a/3$, so $g'(\rho)=3\rho^2+a=0$, and
$g(\rho)=\rho(\rho^2+a)+b=\rho\cdot\tfrac{2a}{3}+b=-b+b=0$.

(ii) A polynomial of degree three cannot have two distinct repeated
roots, since their multiplicities would add up to at least four. The
Frobenius map $x\mapsto x^p$ permutes the roots of $g$ in $\overline{\Fp}$
and preserves multiplicities, so it fixes the unique repeated root
$\rho$; hence $\rho\in\Fp$. Write $g=(x-\rho)^2(x-\sigma)$ with
$\sigma\in\overline{\Fp}$. Comparing the coefficients of $x^2$ gives
$2\rho+\sigma=0$, so $\sigma=-2\rho\in\Fp$; comparing the coefficients
of $x$ and the constant terms gives
$a=\rho^2+2\rho\sigma=-3\rho^2$ and $b=-\rho^2\sigma=2\rho^3$.

(iii) $\rho$ is a triple root if and only if $\sigma=\rho$, that is
$3\rho=0$, that is $\rho=0$ because $p\ne3$. If $\rho=0$ then $a=b=0$
by (ii); if $a=b=0$ then $g=x^3$ and $\rho=0$. If $a\ne0$, then
$\rho\ne0$ by (ii), and $-3b/(2a)=-3\cdot2\rho^3/(2\cdot(-3\rho^2))=\rho$.
If $a=0$ then $\rho=0$, so $b=0$.

(iv) By \eqref{eq:count} the count is $p+\sum_{x\in\Fp}\chi(g(x))$. For
$x\ne\rho$ we have $\chi(g(x))=\chi((x-\rho)^2)\chi(x-\sigma)=\chi(x-\sigma)$,
and $\chi(g(\rho))=0$. Since $x\mapsto x-\sigma$ permutes $\Fp$, and
$\chi$ takes the value $1$ on the $(p-1)/2$ nonzero squares, the value
$-1$ on the $(p-1)/2$ nonsquares and $0$ at $0$, we get
$\sum_{x\in\Fp}\chi(x-\sigma)=0$. Therefore
$\sum_x\chi(g(x))=-\chi(\rho-\sigma)=-\chi(3\rho)$. For a triple root
$3\rho=0$ and the count is $p$. For a double root $\rho\ne0$ and, by
(ii), $-2ab=-2(-3\rho^2)(2\rho^3)=12\rho^5=(3\rho)(2\rho^2)^2$, so
$\chi(-2ab)=\chi(3\rho)$.
\end{proof}

\begin{lemma}[the shift to a critical point]\label{lem:shift}
Let $A,B\in\Zp$, $f(x)=x^3+Ax+B$, $\Delta=4A^3+27B^2$, and let
$\tilde\rho\in\Zp$ satisfy $3\tilde\rho^{\,2}+A=0$. Put $d=3\tilde\rho$
and $c=f(\tilde\rho)=B-2\tilde\rho^{\,3}$. Then
\[
  f(\tilde\rho+u)=u^3+d\,u^2+c
  \qquad\text{and}\qquad
  \Delta=27\,c\,(4\tilde\rho^{\,3}+c).
\]
If moreover $\tilde\rho$ is a unit and $p\mid c$, then
$4\tilde\rho^{\,3}+c$ is a unit and $\vp(c)=\vp(\Delta)$.
\end{lemma}

\begin{proof}
Expanding,
$f(\tilde\rho+u)=(\tilde\rho^{\,3}+A\tilde\rho+B)+(3\tilde\rho^{\,2}+A)u
+3\tilde\rho\,u^2+u^3$, and the coefficient of $u$ vanishes by
hypothesis; the constant term is $f(\tilde\rho)=\tilde\rho^{\,3}
-3\tilde\rho^{\,3}+B=B-2\tilde\rho^{\,3}$. For the identity, $A^3=-27\tilde\rho^{\,6}$
and $B=2\tilde\rho^{\,3}+c$ give
$4A^3+27B^2=-108\tilde\rho^{\,6}+27(4\tilde\rho^{\,6}+4\tilde\rho^{\,3}c+c^2)
=27c(4\tilde\rho^{\,3}+c)$. If $\tilde\rho$ is a unit and $p\mid c$, then
$4\tilde\rho^{\,3}+c\equiv4\tilde\rho^{\,3}\not\equiv0\pmod p$, and $27$ is
a unit, so $\vp(\Delta)=\vp(c)$.
\end{proof}

{\bf Hensel's lemma.} We use it in two forms: the classical one-variable
form, which produces $p$-adic roots, and a two-variable counting form,
which counts the lifts of a smooth point of a plane curve. The $p$-adic
statement is \cite[Ch.~II, \S2.2]{Serre73}; the counting form is what its
proof actually shows.

\begin{lemma}[Hensel, one variable]\label{lem:hensel1}
Let $g\in\Zp[X]$ and $x_0\in\Zp$ with $g(x_0)\equiv0\pmod p$ and
$g'(x_0)\not\equiv0\pmod p$. Then for every $j\ge1$ the congruence
$g(x)\equiv0\pmod{p^j}$ has exactly one solution $x$ modulo $p^j$ with
$x\equiv x_0\pmod p$. These solutions are compatible under reduction, and
they define the unique $\xi\in\Zp$ with $g(\xi)=0$ and $\xi\equiv x_0\pmod p$.
\end{lemma}

\begin{proof}
Induction on $j$; the case $j=1$ is the hypothesis. Let $x_j$ be the
unique solution modulo $p^j$ lifting $x_0$. A solution modulo $p^{j+1}$
lifting $x_0$ reduces modulo $p^j$ to a solution lifting $x_0$, hence to
$x_j$; so it is of the form $x_j+p^js$ with $0\le s<p$. Taylor's formula
gives $g(x_j+p^js)\equiv g(x_j)+p^js\,g'(x_j)\pmod{p^{2j}}$, and
$2j\ge j+1$. Writing $g(x_j)=p^jc$, the condition modulo $p^{j+1}$ is
$c+s\,g'(x_j)\equiv0\pmod p$. As $g'(x_j)\equiv g'(x_0)\not\equiv0
\pmod p$, exactly one $s$ works. The compatible sequence $(x_j)_j$ is an
element $\xi$ of $\Zp=\varprojlim\Z/p^j\Z$ with $g(\xi)\equiv0$ modulo
every $p^j$, so $g(\xi)=0$; any root $\equiv x_0\pmod p$ reduces to $x_j$
for every $j$, so $\xi$ is unique.
\end{proof}

\begin{corollary}\label{cor:hensel1}
\begin{itemize}
\item[(a)] A unit $u\in\Zp^{\times}$ is a square in $\Zp$ if and only if
$\chi(u)=1$, and then it has exactly two square roots, one in each of the
two residue classes $\pm v$ with $v^2\equiv u\pmod p$. In particular
every unit $u\equiv1\pmod p$ has a unique square root $\equiv1\pmod p$.
\item[(b)] Let $A,B\in\Zp$ with $p\mid\Delta$ and $p\nmid A$, and let
$\rho\in\Fp$ be the double root of $f$ modulo $p$. Then $3X^2+A$ has a
unique root $\tilde\rho\in\Zp$ with $\tilde\rho\equiv\rho\pmod p$; it is
a unit, and $\vp(f(\tilde\rho))=\vp(\Delta)$.
\end{itemize}
\end{corollary}

\begin{proof}
(a) If $u=v^2$ in $\Zp$ then $u$ is a square modulo $p$, so $\chi(u)=1$.
Conversely, if $\chi(u)=1$ pick $v$ with $v^2\equiv u\pmod p$; then
$v\not\equiv0$, so $g(X)=X^2-u$ has $g'(v)=2v\not\equiv0\pmod p$, and
Lemma \ref{lem:hensel1} gives a unique root in each of the classes $v$ and
$-v$; a square root of $u$ is a root of $g$, hence reduces to $\pm v$.
(b) By Lemma \ref{lem:cubic}(ii),(iii) applied to $f$ modulo $p$,
$\rho\ne0$ and $3\rho^2+A\equiv0$. For $g(X)=3X^2+A$ we have
$g'(\rho)=6\rho\not\equiv0\pmod p$, so Lemma \ref{lem:hensel1} applies;
the root $\tilde\rho$ is a unit because $\rho\ne0$. Finally
$f(\tilde\rho)\equiv f(\rho)\equiv0\pmod p$, so Lemma \ref{lem:shift}
gives $\vp(f(\tilde\rho))=\vp(\Delta)$.
\end{proof}

\begin{lemma}[Hensel, counting form]\label{lem:hensel}
Let $H\in\Zp[x,y]$, $j\ge1$, and let $(x_0,y_0)$ be a solution of
$H\equiv0\pmod{p^j}$ at which
$\bigl(\partial H/\partial x,\;\partial H/\partial y\bigr)(x_0,y_0)
\not\equiv(0,0)\pmod p$. Then among the $p^2$ pairs
$(x_0+p^js,\;y_0+p^jt)$ with $0\le s,t<p$, exactly $p$ are solutions of
$H\equiv0\pmod{p^{j+1}}$. Consequently, a solution modulo $p$ at which the
gradient of $H$ does not vanish modulo $p$ has exactly $p^{j-1}$ lifts to
solutions modulo $p^j$, for every $j\ge1$.
\end{lemma}

\begin{proof}
Taylor's formula in two variables gives
\[
  H(x_0+p^js,\;y_0+p^jt)\;\equiv\;H(x_0,y_0)\;+\;p^j\Bigl(
  \tfrac{\partial H}{\partial x}(x_0,y_0)\,s+\tfrac{\partial H}{\partial y}(x_0,y_0)\,t\Bigr)
  \pmod{p^{2j}},
\]
because all further terms carry a factor $p^{2j}$ and $p$-adic integer
coefficients; and $2j\ge j+1$. Writing $H(x_0,y_0)=p^jc$, the condition
modulo $p^{j+1}$ is the linear equation
$c+\tfrac{\partial H}{\partial x}(x_0,y_0)\,s+\tfrac{\partial H}{\partial y}(x_0,y_0)\,t\equiv0\pmod p$
in $(s,t)\in\Fp^2$. Its coefficient vector is nonzero, so it has exactly
$p$ solutions. For the consequence, induct on $j$: the gradient of $H$ at
a lift is congruent modulo $p$ to the gradient at $(x_0,y_0)$, since it
depends only on the residues modulo $p$; so each of the $p^{j-1}$ lifts
modulo $p^j$ has exactly $p$ lifts modulo $p^{j+1}$, and lifts of
distinct residues are distinct.
\end{proof}

For a plane curve modulo $p^k$, Hensel's lemma determines the number of
lifts of each nonsingular point modulo $p$. The remaining contribution
comes from the singular points, as follows.

\begin{corollary}[structure of the count modulo $p^k$]\label{cor:structure}
Let $A,B\in\Zp$, $f(x)=x^3+Ax+B$, $F(x,y)=y^2-f(x)$, and for $k\ge1$ let
$N_k$ be the number of solutions of $F\equiv0\pmod{p^k}$ in
$(\Z/p^k\Z)^2$. Let $S\subseteq\Fp^2$ be the set of singular points of the
reduced curve, that is, of the solutions $(x,y)$ modulo $p$ of $F\equiv0$
with $f'(x)\equiv0$ and $2y\equiv0$. For $P\in S$ let $M_k(P)$ be the
number of solutions modulo $p^k$ that reduce to $P$ modulo $p$. Then
\[
  N_k\;=\;p^{k-1}\bigl(N_1-|S|\bigr)\;+\;\sum_{P\in S}M_k(P)
  \qquad(k\ge1).
\]
Moreover $S=\emptyset$ if $p\nmid\Delta$, and $S=\{(\rho,0)\}$ with
$\rho$ the repeated root of $f$ modulo $p$ if $p\mid\Delta$. Hence
\[
  p\nmid\Delta:\quad N_k=p^{k-1}N_1;
  \qquad\qquad
  p\mid\Delta:\quad N_k=p^{k-1}(N_1-1)+M_k,\ \ M_1=1,
\]
where $M_k=M_k((\rho,0))$.
\end{corollary}

\begin{proof}
Every solution modulo $p^k$ reduces to a solution modulo $p$; group the
solutions by their reduction. The gradient of $F$ is $(-f'(x),2y)$. At a
solution $P\notin S$ it does not vanish modulo $p$, so $P$ has exactly
$p^{k-1}$ lifts by Lemma \ref{lem:hensel}; there are $N_1-|S|$ such $P$.
This gives the displayed formula. A point of $S$ has $2y\equiv0$, hence
$y\equiv0$ as $p$ is odd, then $F\equiv0$ gives $f(x)\equiv0$, and
together with $f'(x)\equiv0$ this says that $x$ is a repeated root of $f$
modulo $p$. By Lemma \ref{lem:cubic}(i),(ii) such a root exists if and
only if $p\mid\Delta$, and then it is unique. Finally $M_1=1$ because
$(\rho,0)$ is its own unique lift modulo $p$.
\end{proof}

\begin{definition}[reduction types]\label{def:reduction}
Let $p\ge5$, $A,B\in\Zp$ and $\Delta=4A^3+27B^2$. The equation
$y^2=x^3+Ax+B$ has at $p$
\begin{itemize}
\item[--] {\em good reduction} if $p\nmid\Delta$: the reduced cubic is
nonsingular and defines an elliptic curve over $\Fp$;
\item[--] {\em bad reduction} if $p\mid\Delta$. Then, by Lemma
\ref{lem:cubic}, the reduced cubic has a unique repeated root $\rho\in\Fp$
and the reduced curve has the unique singular point $(\rho,0)$. Two cases:
\begin{itemize}
\item[--] a {\em node} (multiplicative reduction) if $p\nmid A$. Then
$\rho=-3B/(2A)\not\equiv0$ is a double root, and, by Lemma \ref{lem:shift}
read modulo $p$, the reduced curve near $(\rho,0)$ is
$y^2=3\rho\,u^2+u^3$ with $u=x-\rho$. Its tangent cone $y^2=3\rho\,u^2$
consists of two distinct lines, which are defined over $\Fp$ if and only
if $3\rho$ is a square. Put
\[
  \eta\;=\;\chi(-2AB)\;=\;\chi(3\rho)\in\{\pm1\}
\]
(Lemma \ref{lem:cubic}(iv)). The node is {\em split} if $\eta=1$ and
{\em non-split} if $\eta=-1$;
\item[--] a {\em cusp} (additive reduction) if $p\mid A$, and then
$p\mid B$ by Lemma \ref{lem:cubic}(iii). Then $\rho=0$, the reduced
equation is $y^2=x^3$, and the tangent cone at the origin is the double
line $y^2=0$. The nonzero points of $y^2=x^3$ over $\Fp$ are exactly
$(t^2,t^3)$ with $t\in\Fp^{\times}$, since such a point has $t=y/x$; there
are $p-1$ of them.
\end{itemize}
\end{itemize}
\end{definition}

\begin{remark}[reduction types: counts, terminology, minimality]\label{rem:reduction}
(1) The counts modulo $p$ are: $N_1=p-a_p$ with $|a_p|\le2\sqrt p$ for
good reduction (Hasse), $N_1=p-\eta$ for a node and $N_1=p$ for a cusp
(Lemma \ref{lem:cubic}(iv)); the last two cases are the second case of
Lemma \ref{lem:bound}. In the language of \cite[III.2.5, VII.5]{Sil09},
the nonsingular projective points form a group of order $p+1-a_p$, $p-1$,
$p+1$ and $p$ in the four cases good, split node, non-split node and cusp.

(2) In \cite[VII.5]{Sil09} the terms good, multiplicative and additive
reduction refer to an elliptic curve over $\mathbb{Q}_p$ and are defined
through a {\em minimal} Weierstra\ss{} equation. Here the classification
applies to the given equation, which need not be minimal, and $\Delta=0$
is allowed. For $p\ge5$, the changes of variables between two short
Weierstra\ss{} equations of the same curve are $x=u^2x'$, $y=u^3y'$,
which replace $(A,B)$ by $(u^{-4}A,u^{-6}B)$ \cite[III.3.1]{Sil09}. With
$u=p$ the equation $y^2=x^3+Ax+B$ becomes $y'^2=x'^3+(A/p^4)x'+B/p^6$, and
$\Delta$ is divided by $p^{12}$. Hence the equation is minimal at $p$,
i.e. no such substitution keeps the coefficients $p$-integral, if and only
if $\vp(A)<4$ or $\vp(B)<6$; equivalently, if and only if $\vp(\Delta)<12$
or $\vp(A)<4$. Indeed, if $\vp(A)\ge4$ and $\vp(B)\ge6$ then
$\vp(4A^3)\ge12$ and $\vp(27B^2)\ge12$, so $\vp(\Delta)\ge12$; conversely,
if $\vp(A)\ge4$ and $\vp(\Delta)\ge12$ then
$\vp(27B^2)=\vp(\Delta-4A^3)\ge12$, so $\vp(B)\ge6$. A nonminimal
equation has $p\mid A,B$, hence cuspidal reduction in the sense of
Definition \ref{def:reduction}, whatever the reduction type of the curve
it defines; this is the source of row 8 of Theorem \ref{thm:cusp} and of
the scaling depth $t$ in Theorem \ref{thm:compression}.

\end{remark}

{\bf Descent.} Above a singular point, the equation forces divisibility
of the coordinates by powers of $p$. Substituting $x=p^aX$, $y=p^bY$
and dividing by $p^c$ gives an equation modulo $p^{k-c}$. The next lemma
counts the choices of higher digits in $X$ and $Y$ that remain free.

\begin{lemma}[descent]\label{lem:descent}
Let $H\in\Zp[x,y]$, $k\ge1$, and let $a,b,c$ be integers with
$0\le a,b\le c\le k$ such that $H(p^aX,p^bY)=p^c\,H_1(X,Y)$ with
$H_1\in\Zp[X,Y]$. Then
\begin{multline*}
  \#\{(x,y)\bmod p^k:\;H(x,y)\equiv0\pmod{p^k},\ p^a\mid x,\ p^b\mid y\}\\
  =\;p^{\,2c-a-b}\cdot
  \#\{(X,Y)\bmod p^{k-c}:\;H_1(X,Y)\equiv0\pmod{p^{k-c}}\},
\end{multline*}
where a count modulo $p^0$ equals $1$. If the congruence
$H\equiv0\pmod{p^k}$ forces $p^a\mid x$ and $p^b\mid y$, the left-hand
side is its full number of solutions.
\end{lemma}

\begin{proof}
The residues $x$ modulo $p^k$ with $p^a\mid x$ are $x=p^aX$ with $X$
modulo $p^{k-a}$, and likewise $y=p^bY$ with $Y$ modulo $p^{k-b}$. For
these, $H(x,y)\equiv0\pmod{p^k}$ if and only if
$p^cH_1(X,Y)\equiv0\pmod{p^k}$, if and only if
$H_1(X,Y)\equiv0\pmod{p^{k-c}}$. The last condition depends only on $X$
and $Y$ modulo $p^{k-c}$, and $p^{k-c}$ divides $p^{k-a}$ and $p^{k-b}$
because $c\ge a,b$. Every pair $(\bar X,\bar Y)$ modulo $p^{k-c}$
therefore corresponds to $p^{(k-a)-(k-c)}\cdot p^{(k-b)-(k-c)}=p^{2c-a-b}$
pairs $(X\bmod p^{k-a},\,Y\bmod p^{k-b})$, and the pairs $(x,y)$ counted
on the left are exactly those above the solutions of
$H_1\equiv0\pmod{p^{k-c}}$. For $k=c$ the condition is empty, and the
count is $p^{2c-a-b}$.
\end{proof}

When the solutions above a fixed point $(x_0,y_0)$ modulo $p$ are counted,
the lemma is applied after the translation $x\mapsto x_0+x$,
$y\mapsto y_0+y$, with $a=b=1$ at least; further divisibilities are
forced by reducing the equation modulo small powers of $p$. The same
argument with one variable shows: if $h(p^aY)=p^ch_1(Y)$ with
$a\le c\le k$, then the number of $y$ modulo $p^k$ with $p^a\mid y$ and
$h(y)\equiv0\pmod{p^k}$ equals $p^{\,c-a}$ times the number of solutions
of $h_1\equiv0\pmod{p^{k-c}}$. Lemma
\ref{lem:descent} is used in the proof of Theorem \ref{thm:node} with
$(a,b,c)=(1,1,2)$, and in the proof of Theorem \ref{thm:cusp} with
$(a,b,c)=(1,1,2)$, $(1,2,3)$, $(2,2,4)$, $(2,3,5)$ and $(2,3,6)$.

{\bf Binary quadratic forms.} The counts above a node reduce to
$Y^2-dU^2$ with $d$ a unit. Above a double root of the residual cubic
in Section~\ref{sec:cusp}, the form is $y^2-pdz^2$. The first form is
split when $d$ is a square and is the norm form of an unramified
quadratic extension of $\mathbb Q_p$ when $d$ is a nonsquare unit. The second is the norm
form of a ramified quadratic extension. The following lemmas give the
finite-field counts and the valuation properties used in both cases.

\begin{lemma}[conics over $\Fp$]\label{lem:conic}
Let $d\in\Fp^{\times}$ and $c\in\Fp$. Then
\[
  \#\{(U,Y)\in\Fp^2:\;Y^2-dU^2=c\}\;=\;
  \begin{cases}
    p-\chi(d), & c\ne0,\\
    1+(p-1)\bigl(1+\chi(d)\bigr), & c=0.
  \end{cases}
\]
\end{lemma}

\begin{proof}
Let $c=0$. If $U=0$ then $Y=0$. For each of the $p-1$ values $U\ne0$ the
equation $Y^2=dU^2$ has $1+\chi(d)$ solutions $Y$. This gives the second
formula.

Let $c\ne0$ and $\chi(d)=1$, say $d=\delta^2$ with $\delta\ne0$. Then
$Y^2-dU^2=(Y-\delta U)(Y+\delta U)$, and the linear map
$(U,Y)\mapsto(Y-\delta U,\,Y+\delta U)$ is a bijection of $\Fp^2$, its
determinant $2\delta$ being nonzero. The equation $st=c$ has exactly $p-1$
solutions: $s\ne0$ arbitrary and $t=c/s$. So the count is $p-1=p-\chi(d)$.

Let $c\ne0$ and $\chi(d)=-1$. Then $\Fp(\sqrt d)$ is the field with $p^2$
elements, and $Y^2-dU^2=\mathrm{N}(Y+U\sqrt d)$, where
$\mathrm{N}(z)=z\cdot z^p=z^{p+1}$ is the norm to $\Fp$. The norm is a
homomorphism from the cyclic group $\Fp(\sqrt d)^{\times}$ of order
$p^2-1$ to $\Fp^{\times}$. Its kernel consists of roots of $z^{p+1}=1$, so
it has at most $p+1$ elements; hence the image has at least
$(p^2-1)/(p+1)=p-1$ elements, that is, the norm is surjective and every
fiber has exactly $p+1$ elements. So the count is $p+1=p-\chi(d)$.
\end{proof}

\begin{lemma}[norm forms over $\Zp$]\label{lem:normform}
Let $d\in\Zp^{\times}$.
\begin{itemize}
\item[(a)] If $\chi(d)=1$, let $\delta\in\Zp^{\times}$ with $\delta^2=d$
(Corollary \ref{cor:hensel1}(a)). The map
$(y,w)\mapsto(y-\delta w,\;y+\delta w)$ is a bijection of $\Zp^2$ which
induces a bijection of $(\Z/p^n\Z)^2$ for every $n\ge1$, and it transforms
$y^2-dw^2$ into the product of the two coordinates.
\item[(b)] If $\chi(d)=-1$, then $\vp(y^2-dw^2)=2\min(\vp(y),\vp(w))$
for all $y,w\in\Zp$.
\item[(c)] $\vp(y^2-pdz^2)=\min\bigl(2\vp(y),\,1+2\vp(z)\bigr)$ for all
$y,z\in\Zp$.
\end{itemize}
\end{lemma}

\begin{proof}
(a) The map is $\Zp$-linear with determinant $2\delta$, a unit, so its
inverse $(s,t)\mapsto\bigl((s+t)/2,\,(t-s)/(2\delta)\bigr)$ also has
coefficients in $\Zp$; both maps induce maps on $(\Z/p^n\Z)^2$, inverse to
each other. The product of the coordinates is
$(y-\delta w)(y+\delta w)=y^2-dw^2$.

(b) If $y=w=0$ both sides are $\infty$. Otherwise let
$t=\min(\vp(y),\vp(w))$ and write $y=p^ty_1$, $w=p^tw_1$, where $y_1,w_1$
are not both divisible by $p$. Then $y^2-dw^2=p^{2t}(y_1^2-dw_1^2)$, and
$y_1^2-dw_1^2\not\equiv0\pmod p$: if $p\nmid w_1$, a congruence
$y_1^2\equiv dw_1^2$ would make $d\equiv(y_1/w_1)^2$ a square, contrary to
$\chi(d)=-1$; if $p\mid w_1$, then $p\nmid y_1$ and
$y_1^2-dw_1^2\equiv y_1^2\not\equiv0$.

(c) The valuations $\vp(y^2)=2\vp(y)$ and $\vp(pdz^2)=1+2\vp(z)$ are
even and odd respectively, so they are different, and the valuation of a
difference of two elements of different valuations is the smaller one.
\end{proof}

\begin{lemma}[square roots modulo $p^k$]\label{lem:sqrt}
Let $k\ge1$ and $c\in\Z$. The number of $y$ modulo $p^k$ with
$y^2\equiv c\pmod{p^k}$ is
\[
  \begin{cases}
    p^{\lfloor k/2\rfloor}, & p^k\mid c,\\
    0, & \nu=\vp(c)<k \text{ odd},\\
    p^{\nu/2}\bigl(1+\chi(c/p^{\nu})\bigr), & \nu=\vp(c)<k \text{ even}.
  \end{cases}
\]
\end{lemma}

\begin{proof}
If $p^k\mid c$, the condition is $y^2\equiv0\pmod{p^k}$, that is
$2\vp(y)\ge k$, that is $p^{\lceil k/2\rceil}\mid y$; there are
$p^{k-\lceil k/2\rceil}=p^{\lfloor k/2\rfloor}$ such residues. Let
$\nu=\vp(c)<k$. For a solution $y$ we have $\vp(y^2-c)\ge k>\nu$, so
$\vp(y^2)=\vp(c)=\nu$, i.e. $2\vp(y)=\nu$. If $\nu$ is odd there is no
solution. If $\nu=2t$, then $\vp(y)=t$, and with $y=p^ty_1$, $c=p^{2t}c_0$,
$c_0$ a unit, the condition is $y_1^2\equiv c_0\pmod{p^{k-2t}}$. By Lemma
\ref{lem:hensel1}, applied to $X^2-c_0$ whose derivative is a unit at
every unit, the solutions $y_1$ modulo $p^{k-2t}$ correspond bijectively
to the solutions modulo $p$, of which there are $1+\chi(c_0)$. By the
one-variable form of Lemma \ref{lem:descent} with $(a,c)=(t,2t)$, each
solution $y_1$ modulo $p^{k-2t}$ accounts for $p^{t}$ residues $y$ modulo
$p^k$. This gives $p^t(1+\chi(c_0))$.
\end{proof}

{\bf Arithmetic-term ingredients.} The results above are turned into
arithmetic terms with the following standard devices, used in Sections
\ref{sec:valuation} and \ref{sec:cusp}. The greatest common divisor
$\gcd(x,y)$ is an arithmetic term
\cite{Mazzanti2002plainbases,prunescu2024numbertheoreticfunctions}.
By Lemma \ref{lem:eulerfermat}(ii), for $0\le u<p$ the number
$e(u)=u^{(p-1)/2}\bmod p$ lies in $\{0,1,p-1\}$ and determines
$\chi(u)$. For natural numbers $x,y$, the equality indicator is
$E(x,y)=1\tsub\bigl((x\tsub y)+(y\tsub x)\bigr)$, the indicator of
$x\ge y$ is $1\tsub(y\tsub x)$, the minimum is
$\min(x,y)=x\tsub(x\tsub y)$, and parity is $x\bmod2$. A negative residue
$-u\bmod p$ is represented by $(p\tsub(u\bmod p))\bmod p$. The valuation
$\vp$ with $p$ a variable is the arithmetic term of Proposition
\ref{prop:vp}.

\section{The trace of Frobenius}\label{sec:tracefp}

Lemma~\ref{lem:cong} shows that $a_p$ is congruent modulo $p$ to one specific
coefficient of the polynomial $(x^3+\bar A x+\bar B)^e$. Lemma
\ref{lem:extract} shows that the expression inside the outer $\md p$ of
$r$ equals exactly that coefficient. Lemma \ref{lem:bound} bounds
$|a_p|$ and allows us to recover the integer $a_p$ from its residue
$r$ when $p\ge17$.

\begin{lemma}\label{lem:cong}
Let $c_k\in\Z$ denote the coefficients of $f(x)^{e}=\sum_{k=0}^{3e}c_k x^k$.
Then $a_p \equiv c_{p-1} \pmod p$.
\end{lemma}

\begin{proof}
By Euler's criterion, $\chi(u)\equiv u^{e}\pmod p$ for every $u\in\Fp$,
including $u=0$. Hence, modulo $p$,
\[
\sum_{x\in\Fp}\chi(f(x)) \;\equiv\; \sum_{x\in\Fp} f(x)^{e}
\;=\; \sum_{k=0}^{3e} c_k \sum_{x\in\Fp} x^{k}.
\]
The power sum $\sum_{x\in\Fp}x^k$ is congruent to $-1$ modulo $p$ when $k>0$
and $(p-1)\mid k$, and is congruent to $0$ modulo $p$ for every other
$k\ge 0$, see Lemma \ref{lem:sumofpowers}. In the
range $1\le k\le 3e=\tfrac{3(p-1)}{2}$, the only multiple of $p-1$ is
$k=p-1$. Therefore $\sum_{x}\chi(f(x)) \equiv -c_{p-1} \pmod p$, and
$a_p\equiv c_{p-1}\pmod p$ by \eqref{eq:count}.
\end{proof}

The congruence of Lemma~\ref{lem:cong} is classical. The residue
$c_{p-1}\bmod p$ is the Hasse invariant of the curve; see \cite[V.4]{Sil09}.
The next lemma turns this coefficient into an arithmetic term, by evaluating
the polynomial $f^e$ at the point $2^w$ and reading the coefficient as one
digit of the resulting integer.

\begin{lemma}\label{lem:extract} Let the prime $p$ be $\geq 3$. Then
$$\left \lfloor \frac{\left (2^{3w}+\bar A\, 2^w + \bar B \right )^{e}}{2^{w(p-1)}}
\right \rfloor \bmod 2^w \;=\; c_{p-1}.$$
\end{lemma}

\begin{proof}
The coefficients $c_k$ are nonnegative, because $\bar A\ge 0$ and
$\bar B\ge 0$. Their sum equals
$f(1)^{e}=(1+\bar A+\bar B)^{e}\le (2p-1)^{(p-1)/2}\le (2p-1)^{p-1}<2^{p^2}$, the last step by Lemma \ref{lem:inequality}, so each $c_k < 2^w$. Evaluation of polynomials at $2^w$ is a ring homomorphism, so
$(2^{3w}+\bar A2^w+\bar B)^{e} = \sum_k c_k 2^{wk}$. Since every $c_k<2^w$,
the right-hand side is the base-$2^w$ representation of this integer, with
digit $c_k$ at position $k$. Division by $2^{w(p-1)}$ followed by reduction
modulo $2^w$ returns the digit at position $p-1$, which is $c_{p-1}$.
\end{proof}

\begin{lemma}\label{lem:bound} Let the prime $p$ be $> 3$.
If $4\bar A^3+27\bar B^2 \not\equiv 0 \pmod p$, then $|a_p|\le 2\sqrt p$.
If $4\bar A^3+27\bar B^2 \equiv 0 \pmod p$, then $|a_p|\le 1$.
\end{lemma}

The second case is also contained in Lemma \ref{lem:cubic}(iv), which
identifies $a_p$ with $\chi(3\rho)$.

\begin{proof}
The first case is Hasse's theorem \cite[V.1]{Sil09}. In the second case $f$
has a repeated root $\rho$ in $\overline{\Fp}$. The root $\rho$ lies in
$\Fp$, for the following reasons. If $\rho$ is a double root, it is the unique
double root of the cubic $f$, so the Frobenius map fixes it, and the third
root $\sigma$ then also lies in $\Fp$. If $\rho$ is a triple root, then
$3\rho=0$ because $f$ has no $x^2$ term, and $p\ne 3$ gives $\rho=0$; in this
case put $\sigma=\rho$. In both cases $f=(x-\rho)^2(x-\sigma)$ with
$\rho,\sigma\in\Fp$. For every $x\ne\rho$ we have $\chi(f(x))=\chi(x-\sigma)$,
and $\chi(f(\rho))=0$. The sum $\sum_{x\in\Fp}\chi(x-\sigma)$ vanishes: by
Euler's criterion and the substitution $u = x-\sigma$ it is congruent to
$\sum_{u \in \Fp} u^{\frac{p-1}{2}}$, which is $0$ modulo $p$ by Lemma
\ref{lem:sumofpowers}, and being a sum of $p$ values from $\{0,\pm 1\}$, at
least one of which is $0$, its absolute value is at most $p-1$; hence it
equals $0$. Therefore
\[
\sum_{x\in\Fp}\chi(f(x)) \;=\; \left ( \sum_{x\in\Fp}\chi(x-\sigma) \right ) - \chi(\rho-\sigma)
\;=\; -\chi(\rho-\sigma) \in \{0,\pm 1\}. \qedhere
\]

\end{proof}

\section{Formula for $\Fp$}\label{sec:formulafp}

\begin{theorem}\label{thm:ztopz}
Let $p\ge 17$ be prime and let $A,B\in\N$. Put $\bar A = A \md p$ and
$\bar B = B \md p$, and recall that
\begin{equation}
r \;=\; \left (\left \lfloor  \frac{ \left (2^{3p^2} + \bar A\, 2^{p^2} + \bar B\right )^{\lfloor p/2\rfloor}}{
 2^{p^2(p-1)}} \right\rfloor \md 2^{p^2}\right ) \md p .
\end{equation}
Then the number $N$ of points $(x,y)\in(\Z/p\Z)^2$ with $y^2=x^3+Ax+B$ equals
\begin{equation}\label{eq:N}
N \;=\; (p \tsub r) \;+\; p\cdot\Bigl(1 \tsub \bigl((4p+1)\tsub r^2\bigr)\Bigr).
\end{equation}
Equivalently, the quantity $a_p := p-N$ equals $r$ if $r^2\le 4p$, and
equals $r-p$ otherwise; in the nonsingular case this is the trace of
Frobenius. The statement holds for every pair $(A,B)$, including the
pairs for which the cubic $x^3+\bar A x+\bar B$ has a repeated root modulo
$p$. The number of projective points of the cubic is $N+1$.
\end{theorem}

\begin{proof}[Proof of Theorem \ref{thm:ztopz}]
By Lemmas \ref{lem:cong} and \ref{lem:extract}, the residue $r$ defined by
\eqref{eq:r} satisfies $r\equiv a_p \pmod p$ and $0\le r<p$. By Lemma
\ref{lem:bound}, $|a_p|\le 2\sqrt p$ in both cases, so $a_p=r$ or $a_p=r-p$.
For $p\ge 17$ we have $\sqrt p>4$, hence $p-2\sqrt p>2\sqrt p$. Consequently
exactly one of the two candidates lies in the interval $[-2\sqrt p,2\sqrt p]$,
and the test $r^2\le 4p$ identifies it. If $a_p\ge 0$, then $r=a_p\le 2\sqrt
p$ and $r^2\le 4p$. If $a_p<0$, then $r=a_p+p\ge p-2\sqrt p>2\sqrt p$ and
$r^2>4p$. For natural numbers $u$, the indicator of the condition $u^2>4p$
equals $1\tsub((4p+1)\tsub u^2)$. Substituting $u=r$ turns this case
distinction into equation \eqref{eq:N}. The projective closure of a
Weierstra\ss{} cubic has exactly one point at infinity, whether the curve is
singular or not. This gives the final claim.
\end{proof}

\begin{remark}[size of the intermediate integers]\label{rem:width}
The digit width $w=p^2$ was chosen for the simplicity of the statement. The
proof of Lemma \ref{lem:extract} only requires $2^w>(2p-1)^{(p-1)/2}$, so
any admissible width of order $p\log p$ may replace $p^2$, at the cost of a
more complicated expression for $w$. With such a width, the largest
intermediate integer of the term, the power in \eqref{eq:r}, shrinks from
about $\tfrac32 p^3$ binary digits to $O(p^2\log p)$ binary digits. This
variant was also verified numerically, for all primes $17\le p\le 199$ and
six parameter pairs for each prime.
\end{remark}

\begin{remark}[restriction to prime moduli]\label{rem:nogo}
The proof of Lemma~\ref{lem:cong} uses the field structure of $\Z/p\Z$ in two
places: Euler's criterion requires it, and so does the evaluation of the
power sums. The argument does not extend to a general modulus $n$ by a change
of exponent, for the following reason. An extension of this type would need
one exponent $e'$ with $u^{e'}\equiv u^{(p-1)/2} \pmod p$ for every $u$ and
every prime $p$ dividing $n$, that is, $e'\equiv (p-1)/2 \pmod{p-1}$ for
every prime $p$ dividing $n$. If $n$ has a prime factor congruent to $1$
modulo $4$, this congruence forces $e'$ to be even. If $n$ also has a prime
factor congruent to $3$ modulo $4$, it forces $e'$ to be odd. For such $n$ no
exponent $e'$ exists. This observation rules out only the single-exponent
extension described above; short arithmetic terms for composite moduli can be
obtained by a different method, as Section~\ref{sec:composite} shows.
\end{remark}

\section{The curve over $\Fp$, counted over every $\mathbb{F}_{p^k}$}\label{sec:formulafpk}

{\bf Notation.} Throughout, $p\ge 17$ is prime, $k\ge 1$, $q=p^k$, and
$A,B\in\N$ with $\bar A=A \md p$, $\bar B=B \md p$ and
$f(x)=x^3+\bar Ax+\bar B$. As in Section \ref{sec:facts}, $N_k$ denotes the number of pairs $(x,y)\in\Fq^2$ with $y^2=f(x)$, so that $N_1=N$. Put $a_p=p-N_1$; for a nonsingular cubic this is the trace of Frobenius. Let $\chi_q$ denote
the quadratic character of $\Fq$, extended by $\chi_q(0)=0$.

For a nonsingular cubic, the count over $\mathbb F_{p^k}$ is determined
by a linear recurrence whose coefficients are $a_p$ and $p$.
For fixed $k$, the recurrence gives a polynomial expression in these
quantities. To obtain a term with $k$ variable, we use the generating
function in Theorem~\ref{thm:term}. For a singular cubic, the count
follows from the restriction of the quadratic character of
$\mathbb F_{p^k}$ to $\Fp$.

\begin{lemma}\label{lem:restr}
For every $c\in\Fp$ with $c\ne 0$, one has $\chi_q(c)=\chi_p(c)^k$.
\end{lemma}

\begin{proof}
By Euler's criterion in $\Fq$, $\chi_q(c)=c^{(q-1)/2}$. Write
\[
\frac{q-1}{2}=\frac{p-1}{2}\,\bigl(p^{k-1}+p^{k-2}+\cdots+1\bigr).
\]
The number $c^{(p-1)/2}=\chi_p(c)$ equals $\pm 1$, and the sum
$p^{k-1}+\cdots+1$ has $k$ odd summands, so it is congruent to $k$ modulo
$2$. Therefore $\chi_q(c)=\chi_p(c)^{\,p^{k-1}+\cdots+1}=\chi_p(c)^k$.
\end{proof}

\begin{theorem}\label{thm:ext}
Define the integer sequence $s_0=2$, $s_1=a_p$,
$s_j=a_p\,s_{j-1}-p\,s_{j-2}$ for $j\ge 2$.
\begin{itemize}
\item[(i)] If $4\bar A^3+27\bar B^2\not\equiv 0 \pmod p$, then
$N_k=p^k-s_k$, and $|s_k|\le 2p^{k/2}$.
\item[(ii)] If $4\bar A^3+27\bar B^2\equiv 0 \pmod p$, then
$N_k=p^k-a_p^{\,k}$, where $a_p\in\{0,\pm1\}$.
\end{itemize}
\end{theorem}

\begin{proof}
(i) The curve is an elliptic curve over $\Fp$. By the rationality of its
zeta function \cite[V.2]{Sil09}, the number of projective points over $\Fq$
equals $q+1-\alpha^k-\beta^k$, where $\alpha,\beta$ are the two roots of
$T^2-a_pT+p$. The sums $s_k=\alpha^k+\beta^k$ are rational integers and
satisfy the stated recurrence, because $\alpha+\beta=a_p$ and
$\alpha\beta=p$. Subtracting the single point at infinity gives
$N_k=q-s_k$. Since $|\alpha|=|\beta|=\sqrt p$, we get $|s_k|\le 2p^{k/2}$.

(ii) As in the singular case with $q=p$,
$f=(x-\rho)^2(x-\sigma)$ with $\rho,\sigma\in\Fp$, and the same
computation, carried out over $\Fq$ instead of $\Fp$, gives
$\sum_{x\in\Fq}\chi_q(f(x))=-\chi_q(\rho-\sigma)$; here the sum
$\sum_{x\in\Fq}\chi_q(x-\sigma)$ vanishes because $x\mapsto x-\sigma$ permutes
$\Fq$, $\chi_q(0)=0$, and $\chi_q$ takes the value $1$ on the $(q-1)/2$ nonzero
squares and the value $-1$ on the remaining $(q-1)/2$ elements of $\Fq^{*}$. Hence
$N_k=q+\sum_x\chi_q(f(x))=q-\chi_q(\rho-\sigma)$. For $k=1$ this
identifies $a_p=\chi_p(\rho-\sigma)$. By Lemma \ref{lem:restr}, either
$\rho=\sigma$ and both sides of (ii) equal $q$, or
$\chi_q(\rho-\sigma)=\chi_p(\rho-\sigma)^k=a_p^{\,k}$.
\end{proof}

The next theorem turns part (i) into a single arithmetic term with $k$ as a
free variable. The sequence $s_j$ has mixed signs, and digit extraction
needs nonnegative digits. C-recursive sequences with mixed signs have been treated in \cite{PS24}. The remedy is to shift $s_j$ by the dominating
geometric sequence $2p^j$. The shifted sequence $t_j=s_j+2p^j$ is
nonnegative by the Weil bound of Theorem \ref{thm:ext}(i), still satisfies
a linear recurrence, of order three instead of two, and $s_k$ is recovered
from $t_k$ at the end by subtracting the shift.

\begin{theorem}\label{thm:term}
Assume $4\bar A^3+27\bar B^2\not\equiv 0\pmod p$. Let $r$ be as in
Theorem \ref{thm:ztopz} and put $i=1\tsub((4p+1)\tsub r^2)$, so that
$a_p=r-p\,i$ by that theorem. Define its positive and negative parts by
\[
a^+=r(1\tsub i),\qquad a^-=i(p\tsub r),
\]
so that $a_p=a^+-a^-$. Put
\[
W=(k+1)\,p,\qquad u=2^W,
\]
\[
P^+=4u^2+3a^-u+p(a^++2),\qquad
P^-=(3a^++2p)u+pa^-,
\]
\[
\widetilde P=P^+\tsub P^-,\qquad
\widetilde Q=\bigl((u^2+a^-u+p)\tsub a^+u\bigr)(u\tsub p),
\]
\[
F=\left\lfloor \frac{u^{\,k+1}\,\widetilde P}{\widetilde Q}
\right\rfloor .
\]
Then
\[
N_k \;=\; 3p^k\tsub\bigl(F \md u\bigr).
\]
In particular, the right-hand side is an arithmetic term over the natural
numbers, with no signed intermediate quantities.
\end{theorem}

\begin{proof}
Put $t_j=s_j+2p^j$. By the bound of Theorem \ref{thm:ext}(i),
$0\le t_j\le 4p^j$, so in particular $t_j<u$ for $j\le k$. The generating
function of $s_j$ is $(2-a_pz)/(1-a_pz+pz^2)$ and that of $2p^j$ is
$2/(1-pz)$; their sum is $P(z)/Q(z)$ with $Q(z)=(1-a_pz+pz^2)(1-pz)$ and
$P(z)=4-(3a_p+2p)z+p(a_p+2)z^2$. Since $|a_p|\le2\sqrt p<p<u$, both
factors of $Q(1/u)$ are positive. Moreover the series converges at $1/u$
and has positive value. Hence $P(1/u)>0$. It follows that the truncated
subtractions in the definitions of $\widetilde P$ and $\widetilde Q$ are
exact, and direct expansion gives
\[
\widetilde P=u^2P(1/u),\qquad \widetilde Q=u^3Q(1/u).
\]
Consequently
\[
\frac{u^{\,k+1}\widetilde P}{\widetilde Q}
= u^{k}\,\frac{P(1/u)}{Q(1/u)}
= \sum_{j\ge 0} t_j\,u^{\,k-j}
= \sum_{j=0}^{k} t_j\,u^{\,k-j} \;+\; \sum_{j> k} t_j\,u^{\,k-j}.
\]
The second sum is smaller than $1$: it is at most
$\sum_{\ell\ge 1}4p^{k+\ell}u^{-\ell}=4p^{k+1}/(u-p)\le 8p^{k+1}/u$, as $u\ge 2p$, and $u=2^{(k+1)p}>8p^{k+1}$
for $p\ge 17$. The first sum is an integer whose base-$u$ digits are the
$t_j$. Hence $F$ equals the first sum, and $F\md u=t_k$. Finally
$N_k=p^k-s_k=3p^k-t_k\ge0$, so the last truncated subtraction in the
statement is also exact.
\end{proof}

\begin{remark}[one term for all cases]\label{rem:oneterm}
Let $N_k^{\rm ns}=3p^k\tsub(F\md u)$ denote the term in Theorem
\ref{thm:term}. The formulas defining it remain well-defined for a singular
cubic as well: Lemma \ref{lem:bound} then gives $|a_p|\le1$, so in particular
$\widetilde Q>0$ by the same argument as in the proof. For the singular case put
\[
h=r\tsub\bigl(r\tsub(p\tsub r)\bigr),\qquad
N_k^{\rm sing}=(p^k\tsub h)+2h\bigl((ik)\md2\bigr).
\]
Indeed, in the singular case $r\in\{0,1,p-1\}$, $h$ is respectively
$0,1,1$, and Theorem \ref{thm:ext}(ii) gives precisely this formula. Now set
\[
\delta=1\tsub\bigl((4\bar A^3+27\bar B^2)\md p\bigr).
\]
The single arithmetic term
\[
(1\tsub\delta)N_k^{\rm ns}+\delta N_k^{\rm sing}
\]
equals $N_k$ for every $A,B\in\N$, since $\delta$ is $1$ exactly in the
singular case. The width $W=(k+1)p$ was chosen for simplicity; any width
with $2^W>8p^{k+1}$ works, so a width of order $k\log p$ suffices if a term
for the binary length is used.
\end{remark}

\section{The curve over $\Z$, counted over every $\mathbb{Z}/{p^k} \mathbb Z$}\label{sec:zoverptokz}

Fix a prime $p\ge5$ and coefficients $A,B\in\N$. Put $f(x)=x^3+Ax+B$ and

\[
  \Delta \;=\; 4A^3+27B^2 \;=\; -\operatorname{disc}(f),
  \qquad
  N_k \;=\; \#\bigl\{(x,y)\in(\Z/p^k\Z)^2 \;:\; y^2=f(x)\bigr\}.
\]

From this section onward, $N_k$ denotes the count modulo $p^k$.
For $p\ge17$, its initial value $N_1=p-a_p$ is given by
Theorem~\ref{thm:ztopz}.

If $p\nmid\Delta$, Hensel's lemma gives $N_k=p^{k-1}N_1$.
If $p\mid\Delta$, the reduced curve has one singular point $(\rho,0)$,
and Corollary~\ref{cor:structure} gives
$N_k=p^{k-1}(N_1-1)+M_k$, where $M_k$ counts the solutions reducing to
$(\rho,0)$. By Lemma~\ref{lem:cubic}, the singular point is a node when
$p\nmid A$, with $\rho\equiv-3B/(2A)\pmod p$. When $p\mid A$, we also
have $p\mid B$, and the singular point is the cusp $(0,0)$.
The formulas below compute $M_k$ in these two cases.

\section{Good reduction}

\begin{theorem}\label{thm:good}
If $p\nmid\Delta$ then $N_k=p^{k-1}N_1$ for every $k\ge1$.
\end{theorem}

\begin{proof}
This is the case $p\nmid\Delta$ of Corollary \ref{cor:structure}.
\end{proof}

\begin{corollary}\label{cor:good}
Let $p\ge17$ be prime with $p\nmid\Delta$, and let $r$ be as in Theorem \ref{thm:ztopz}, i.e.
\[
  r \;=\;\left\lfloor\frac{\bigl(2^{3p^2}+\bar A\,2^{p^2}+\bar B\bigr)^{\lfloor p/2\rfloor}}{2^{p^2(p-1)}}\right\rfloor
      \bmod 2^{p^2} \bmod p .
\]
Then for every $k\ge1$
\[
  N_k \;=\; p^{\,k-1}\Bigl(\,(p\tsub r)+p\cdot\bigl(1\tsub((4p+1)\tsub r^2)\bigr)\Bigr),
\]
an arithmetic term in $A,B,p,k$.
\end{corollary}

This is the term of Theorem \ref{thm:ztopz} multiplied by $p^{k-1}$.

\section{Bad reduction: the node}

Assume now $p\mid\Delta$. By Corollary \ref{cor:structure},
\begin{equation}\label{eq:split}
  N_k \;=\; p^{k-1}\,(N_1-1)\;+\;M_k,
\end{equation}
where $M_k$ counts the points of $(\Z/p^k\Z)^2$ on the curve reducing to
$(\rho,0)$. The two cases of Definition \ref{def:reduction} are the \emph{node}
($p\nmid A$) and the \emph{cusp} ($p\mid A$, hence $p\mid B$). This section
treats the node.

\subsection{Invariants}

Let $p\nmid A$. Since $\rho\not\equiv0$, the polynomial $3X^2+A$ has the simple
root $\rho$ modulo $p$, so by Corollary \ref{cor:hensel1}(b) it has a unique root
$\tilde\rho\in\Zp^{\times}$ lifting $\rho$. Put
\[
  d \;=\; 3\tilde\rho \;\in\;\Zp^{\times},
  \qquad
  a_0 \;=\; f(\tilde\rho) \;=\; B-2\tilde\rho^{\,3},
  \qquad
  m \;=\; v_p(a_0)\;\ge\;1 .
\]
Substituting $x=\tilde\rho+u$ kills the linear term:
\begin{equation}\label{eq:shift}
  f(\tilde\rho+u)\;=\;u^3+d\,u^2+a_0 .
\end{equation}
Both $m$ and the class of $d$ are computable from $A,B,p$ alone. Indeed,
\eqref{eq:shift} is Lemma \ref{lem:shift}, which also gives
$\Delta=27\,a_0\,(4\tilde\rho^{\,3}+a_0)$ with $4\tilde\rho^{\,3}+a_0$ a
unit, so that
\[
  m \;=\; v_p(\Delta),
  \qquad\text{and, by Lemma \ref{lem:cubic}(iv),}\qquad
  \eta \;:=\; \legendre{d} \;=\; \legendre{3\rho} \;=\; \legendre{-2AB}.
\]
This $\eta$ is the invariant of Definition \ref{def:reduction}: the node is
split if $\eta=1$ and non-split if $\eta=-1$.

Only $\min(m,k)$ will matter, and $\min(m,k)=v_p\bigl((\Delta\bmod p^k)+p^k\bigr)$,
which incidentally avoids $v_p(0)$ when $\Delta=0$.

\subsection{Closed form}

\begin{theorem}\label{thm:node}
Let $p\ge5$, $p\mid\Delta$, $p\nmid A$, and let $m=v_p(\Delta)\in\{1,2,\dots\}\cup\{\infty\}$
and $\eta=\legendre{-2AB}\in\{\pm1\}$. Then for every $k\ge1$,
\[
  \eta=+1:\qquad
  M_k=\begin{cases}
    p^{\,k-1}\bigl((k-1)p-(k-2)\bigr), & k\le m,\\[2pt]
    p^{\,k-1}(m-1)(p-1),               & k>m;
  \end{cases}
\]
\[
  \eta=-1:\qquad
  M_k=\begin{cases}
    p^{\,2\lfloor k/2\rfloor}, & k\le m,\\[2pt]
    0,                         & k>m,\ m \text{ odd},\\[2pt]
    p^{\,k-1}(p+1),            & k>m,\ m \text{ even}.
  \end{cases}
\]
\end{theorem}

In particular $m=1$ gives $M_k=0$ for all $k\ge2$: the singular point does not
lift at all. Together with \eqref{eq:split} and Theorem \ref{thm:ztopz} this
determines $N_k$.

\subsection{Proof of Theorem \ref{thm:node}}

Fix a representative $\tilde\rho$ modulo $p^k$. By \eqref{eq:shift},
\[
  M_k=\#\bigl\{(u,y) \bmod p^k \;:\; p\mid u,\ p\mid y,\
      y^2\equiv u^3+du^2+a_0 \pmod{p^k}\bigr\}.
\]
For $k=1$ only $(0,0)$ occurs, so $M_1=1$. Let $k\ge2$ and write $u=pU$,
$y=pY$ with $U,Y$ modulo $p^{k-1}$. The congruence becomes
\[
  p^2\bigl(Y^2-dU^2-pU^3\bigr)\;\equiv\;a_0 \pmod{p^k}.
\]
If $m=1$ this is impossible, so $M_k=0$. If $m\ge2$, write $a_0=p^2a_0'$; the
condition reads $Q(U,Y)\equiv a_0'\pmod{p^{k-2}}$ with
$Q(U,Y)=Y^2-dU^2-pU^3$, and involves $(U,Y)$ only modulo $p^{k-2}$. Hence,
by Lemma \ref{lem:descent} with $(a,b,c)=(1,1,2)$,
\begin{equation}\label{eq:MC}
  M_k \;=\; p^2\,C_{k-2}(a_0'),
  \qquad
  C_j(c):=\#\bigl\{(U,Y)\bmod p^j : Q(U,Y)\equiv c \pmod{p^j}\bigr\},\quad C_0:=1,
\end{equation}
and $v_p(a_0')=m-2$.

\smallskip
\noindent\emph{Recursion for $C_j$.} Let $j\ge1$ and $\nu=v_p(c)$; put
$s=(p-1)(1+\eta)$. Split the count by $(U,Y)\bmod p$.

\emph{(i) $(U,Y)\not\equiv(0,0)$.} Here
$\nabla Q\equiv(-2dU,\,2Y)\not\equiv(0,0)\pmod p$, since $d$ is a unit and $p$
is odd. Lemma \ref{lem:hensel}, applied to $Q-c$, gives $p^{j-1}$
solutions modulo $p^j$ above each such solution modulo $p$. Modulo $p$ we have
$Q\equiv Y^2-dU^2$, and the number of its solutions with $(U,Y)\ne(0,0)$ is
\[
  p-\eta \quad(\nu=0),
  \qquad\qquad
  s=(p-1)(1+\eta) \quad(\nu\ge1).
\]
Both are Lemma \ref{lem:conic}: for $\nu=0$ the right-hand side is a
nonzero constant; for $\nu\ge1$ the conic is $Y^2=dU^2$, with
$1+(p-1)(1+\eta)$ solutions, and discarding the origin leaves $s$.

\emph{(ii) $(U,Y)\equiv(0,0)$.} Put $U=pU'$, $Y=pY'$ with $U',Y'$ modulo
$p^{j-1}$; then $Q=p^2\bigl(Y'^2-dU'^2-p^2U'^3\bigr)$. For $j\le2$ the left side
vanishes modulo $p^j$, so the condition is $p^j\mid c$ and then all
$p^{2(j-1)}$ pairs qualify. For $j\ge3$ the condition is $p^2\mid c$ together
with $Y'^2-dU'^2-p^2U'^3\equiv c/p^2 \pmod{p^{j-2}}$, giving $p^2$ times the
corresponding count (Lemma \ref{lem:descent} with $(a,b,c)=(1,1,2)$). That count is again $C_{j-2}$: the cubic term and its
gradient vanish modulo $p$ at every stage, so replacing $pU^3$ by $p^{e}U^3$
with $e\ge1$ changes neither (i) nor (ii), and induction on $j$ gives the same
function.

Collecting, for $j\ge1$:
\begin{equation}\label{eq:rec}
  C_j(\nu)=
  \begin{cases}
    p^{j-1}(p-\eta), & \nu=0,\\
    s+1, & \nu\ge1,\ j=1,\\
    p^{j-1}s, & \nu=1,\ j\ge2,\\
    p^{j-1}s+p^2, & \nu\ge2,\ j=2,\\
    p^{j-1}s+p^2\,C_{j-2}(\nu-2), & \nu\ge2,\ j\ge3.
  \end{cases}
\end{equation}
Only $\min(\nu,j)$ occurs, so $C_j(\nu)$ depends only on $\min(\nu,j)$.

\smallskip
\noindent\emph{Case $\eta=+1$, $s=2(p-1)$.} For $\nu\ge j$ set
$c_j=C_j(\nu)$. Then $c_1=2p-1$, $c_2=p(3p-2)$, and
$c_j=p^{j-1}s+p^2c_{j-2}$; induction gives
$c_j=p^{j-1}\bigl((j+1)p-j\bigr)$. For $0\le\nu<j$, induction on $\nu$ using
\eqref{eq:rec} gives $C_j(\nu)=p^{j-1}(\nu+1)(p-1)$: the cases $\nu=0,1$ are
immediate, and for $\nu\ge2$
\[
  p^{j-1}s+p^2\cdot p^{j-3}(\nu-1)(p-1)=p^{j-1}(p-1)\bigl(2+\nu-1\bigr).
\]
Now \eqref{eq:MC} with $j=k-2$, $\nu=m-2$ gives, for $k\le m$,
$M_k=p^2p^{k-3}\bigl((k-1)p-(k-2)\bigr)$, and for $k>m$,
$M_k=p^2p^{k-3}(m-1)(p-1)$. Both formulas also hold at $k=1,2$ and cover
$m=1$.

\smallskip
\noindent\emph{Case $\eta=-1$, $s=0$.} For $\nu\ge j$, \eqref{eq:rec} gives
$c_1=1$, $c_2=p^2$, $c_j=p^2c_{j-2}$, so $c_j=p^{2\lfloor j/2\rfloor}$. For
$\nu<j$: $C_j(0)=p^{j-1}(p+1)$, $C_j(1)=0$, and $C_j(\nu)=p^2C_{j-2}(\nu-2)$
for $\nu\ge2$, so $C_j(\nu)=p^{j-1}(p+1)$ for $\nu$ even and $0$ for $\nu$ odd.
Substituting in \eqref{eq:MC} yields the stated formulas, using
$2+2\lfloor(k-2)/2\rfloor=2\lfloor k/2\rfloor$. \qed

\subsection{Second proof of Theorem \ref{thm:node}}

 The proof above evaluates the recursion \eqref{eq:rec}. The
formulas can also be read off directly from the quadratic form
$Y^2-dU^2$, using Lemma \ref{lem:normform} and Lemma \ref{lem:sqrtmap} of
Section \ref{sec:cusp}, whose proof does not depend on the present
section. Recall from \eqref{eq:MC} that $M_1=1$, $M_k=0$ for $k\ge2$ if
$m=1$, and otherwise $M_k=p^2C_{k-2}(a_0')$ with $\vp(a_0')=m-2$, where
\[
 C_j(b)=\#\bigl\{(U,Y)\bmod p^j:\;Y^2-dU^2-pU^3\equiv b\pmod{p^j}\bigr\},
 \qquad C_0=1 .
\]
Lemma \ref{lem:sqrtmap} replaces $dU^2+pU^3$ by $dW^2$ through a bijection
$U\mapsto W$ modulo $p^j$, so
$C_j(b)=\#\{(W,Y)\bmod p^j:\;Y^2-dW^2\equiv b\}$. For $j\ge1$ and
$\nu=\vp(b)$ the counts are
\[
\begin{array}{c|cc}
 &\nu\ge j&\nu<j\\ \hline
\eta=1&p^{j-1}((j+1)p-j)&p^{j-1}(\nu+1)(p-1)\\[2pt]
\eta=-1&p^{2\lfloor j/2\rfloor}&
 \begin{cases}p^{j-1}(p+1),&\nu\text{ even},\\0,&\nu\text{ odd}.
 \end{cases}
\end{array}
\]
For $\eta=1$, factor as $(Y-\sqrt d\,W)(Y+\sqrt d\,W)=b$; the linear
change is a bijection modulo $p^j$ by Lemma \ref{lem:normform}(a), so
$C_j(b)$ is the number of pairs $(s,t)$ modulo $p^j$ with $st\equiv b$.
When $\nu<j$, each possible valuation $0,\ldots,\nu$ of $s$ contributes
$(p-1)p^{j-1}$ pairs: $s$ has $(p-1)p^{j-1-\vp(s)}$ choices, and then $t$
is determined modulo $p^{j-\vp(s)}$, that is, has $p^{\vp(s)}$ choices.
When $\nu\ge j$, the valuations $0,\ldots,j-1$ of $s$ each contribute
$(p-1)p^{j-1}$ in the same way, and $s\equiv0$ contributes $p^j$. This
gives the split column.

For $\eta=-1$, $\vp(Y^2-dW^2)=2\min(\vp(Y),\vp(W))$ by Lemma
\ref{lem:normform}(b), and a nonzero right-hand side modulo $p$ has $p+1$
solutions by Lemma \ref{lem:conic}. Lemma \ref{lem:hensel} and Lemma
\ref{lem:descent} with $(a,b,c)=(1,1,2)$ then give the non-split column,
including the count for $b\equiv0$. Substituting the table in
$M_k=p^2C_{k-2}(a_0')$ gives the formulas of Theorem \ref{thm:node};
$k=1,2$ are checked directly.

\section{The valuation as an arithmetic term}\label{sec:valuation}

The nodal and cuspidal formulas involve $p$-adic valuations.
We give an arithmetic term for $v_p(n)$, valid for every prime $p$
and every integer $n\ge1$.

\begin{proposition}\label{prop:vp}
For $n\ge1$ and a prime $p\ge2$,
\[
  v_p(n)\;=\;\left\lfloor
    \frac{\gcd(n,p^n)^{\,n+1} \bmod \bigl(p^{\,n+1}-1\bigr)^2}{p^{\,n+1}-1}
  \right\rfloor .
\]
\end{proposition}

\begin{proof}
Let $v=v_p(n)$, so $\gcd(n,p^n)=p^{v}$ and $v\le n$. With $b=p^{\,n+1}$ we get
$\gcd(n,p^n)^{\,n+1}=b^{v}=(1+(b-1))^{v}\equiv1+v(b-1)\pmod{(b-1)^2}$, and
$0\le 1+v(b-1)<(b-1)^2$ because $v\le n\le b-2$. Dividing by $b-1$ and taking
the floor returns $v$.
\end{proof}

The primality of $p$ is essential: the proof uses $\gcd(n,p^n)=p^{\vp(n)}$,
which fails for composite bases. For $C\ge0$ and $L\ge0$ the capped
valuation has the representation with positive input
\begin{equation}\label{eq:cap}
 \min(\vp(C),L)=\vp\bigl((C\bmod p^L)+p^L\bigr):
\end{equation}
if $\vp(C)<L$, adding $p^L$ does not change the valuation of $C\bmod p^L$,
and otherwise the input is exactly $p^L$. In particular $\vp(0)$ never
occurs.

The greatest common divisor is an arithmetic term
\cite{Mazzanti2002plainbases,prunescu2024numbertheoreticfunctions}.
Euler's criterion encodes the Legendre symbol by the residue
$a^{(p-1)/2}\bmod p$, as described in Section~\ref{sec:facts-zpk}.
The indicators $1\tsub(\Delta\bmod p)$ and $1\tsub(A\bmod p)$ select
the reduction type. Together with the capped valuation and the
equality, inequality and parity indicators, they express the
good-reduction and nodal formulas as arithmetic terms.

\section{Bad reduction: the cusp}\label{sec:cusp}
Fix a prime $p\ge5$. For $A,B\in\Z$ and $k\ge1$, put
\[
 N_k(A,B)=\#\{(x,y)\bmod p^k:y^2\equiv x^3+Ax+B\pmod{p^k}\},
 \qquad N_0(A,B)=1.
\]

Write $\Delta=4A^3+27B^2$, $\chi(u)=\legendre{u}$, and $\vp(0)=\infty$.
The formulas include nonminimal equations and $\Delta=0$.
For the arithmetic-term statement in Corollary~\ref{cor:term}, we
restrict the coefficients to $A,B\ge0$.

Assume henceforth in the cusp case that $p\mid A,B$. By Definition
\ref{def:reduction} the reduction is $y^2=x^3$, with the $p-1$ non-singular
points $(t^2,t^3)$, $t\in\Fp^{\times}$, and the singular point $(0,0)$.
If $M_k$ denotes the count above this point, then by Corollary
\ref{cor:structure}
\begin{equation}\label{eq:split1}
 N_k(A,B)=(p-1)p^{k-1}+M_k,\qquad M_1=1.
\end{equation}

\subsection{The complete cusp table}\label{sec:table}
One auxiliary count is needed. For $a,b\in\Z$ not both divisible by $p$, set
\begin{equation}\label{eq:Kdef}
 K_\ell(a,b)=\#\{(X,Y)\bmod p^\ell:X^3+aX+b\equiv pY^2\pmod{p^\ell}\},
 \qquad K_0=1.
\end{equation}
Theorem~\ref{thm:K} expresses $K_\ell(a,b)$ in terms of the
roots of $X^3+aX+b$ modulo $p$, the valuation of its discriminant and
a quadratic character.

\begin{theorem}[cuspidal reduction]\label{thm:cusp}
Let $\alpha=\vp(A)\ge1$ and $\beta=\vp(B)\ge1$. Exactly one row of the
following table applies. In that row,
\[
 M_1=1,\qquad M_k=p^k\quad(2\le k<h),
\]
and the last column gives $M_k$ for every $k\ge h$.
\begin{center}
\small
\renewcommand{\arraystretch}{1.55}
\begin{tabular}{@{}c p{6.0cm} c l@{}}
\toprule
Row & Conditions & $h$ & $M_k$ for $k\ge h$\\
\midrule
1 & $\beta=1$ & 2 & $0$\\
2 & $\alpha=1,\ \beta\ge2$ & 2 & $p^k$\\
3 & $\alpha\ge2,\ \beta=2$ & 3 & $(1+\chi(B/p^2))p^k$\\
4 & $\alpha\ge2,\ \beta\ge3$, and ($\alpha=2$ or $\beta=3$) & 4
  & $p^3K_{k-3}(A/p^2,B/p^3)$\\
5 & $\alpha\ge3,\ \beta=4$ & 5 & $(1+\chi(B/p^4))p^k$\\
6 & $\alpha=3,\ \beta\ge5$ & 5 & $p^k$\\
7 & $\alpha\ge4,\ \beta=5$ & 6 & $0$\\
8 & $\alpha\ge4,\ \beta\ge6$ & 6 & $p^7N_{k-6}(A/p^4,B/p^6)$\\
\bottomrule
\end{tabular}
\end{center}
The single recurrence in row 8 is removed by Theorem~\ref{thm:compression}.
\end{theorem}

\subsection{Proof of the cusp table}
Every substitution below is an instance of Lemma \ref{lem:descent}; the
exponents $(a,b,c)$ are indicated, and the factor $p^{2c-a-b}$ counts the
unused higher digits. All counts of an auxiliary equation modulo $p^j$ use both of its
variables modulo $p^j$.

\textit{First descent: rows 1--3.}
Above $(0,0)$, write $x=pX$, $y=pY$. For $k\ge2$ the equation is
\[
 p^2Y^2\equiv p^3X^3+pAX+B\pmod{p^k}.
\]
If $\beta=1$, this is impossible, proving row 1. Otherwise division by $p^2$
gives
\begin{equation}\label{eq:H1}
 Y^2\equiv pX^3+(A/p)X+B/p^2\pmod{p^{k-2}}.
\end{equation}
By Lemma \ref{lem:descent} with $(a,b,c)=(1,1,2)$, each solution of \eqref{eq:H1} has $p^2$ choices of the unused higher digits.
In particular $M_2=p^2$. If $\alpha=1$, the derivative in $X$ is a unit,
and modulo $p$ there are exactly $p$ solutions, one $X$ for each $Y$.
Consequently $M_k=p^2p^{k-2}=p^k$, proving row 2.
If $\alpha\ge2$ and $\beta=2$, the reduction of \eqref{eq:H1} is
$Y^2=B/p^2\ne0$. It has $p(1+\chi(B/p^2))$ smooth points. Lemma \ref{lem:hensel}
proves row 3, including the possibility of no points.

\textit{Second descent: the residual cubic.}
Suppose $\alpha\ge2$, $\beta\ge3$, and $k\ge3$. Equation~\eqref{eq:H1}
forces its $Y$ to be divisible by $p$. Thus, in original coordinates,
$x=pX$, $y=p^2Y$, and division by $p^3$ gives
\begin{equation}\label{eq:Kstage}
 X^3+(A/p^2)X+B/p^3\equiv pY^2\pmod{p^{k-3}}.
\end{equation}
This is Lemma \ref{lem:descent} with $(a,b,c)=(1,2,3)$: the multiplicity
over the modulus $p^{k-3}$ is $p^{6-3}=p^3$.
At $k=3$ the condition is vacuous, hence $M_3=p^3$.
If $\alpha=2$ or $\beta=3$, the reduced cubic does not have a triple root,
by Lemma \ref{lem:cubic}(iii). This proves row 4 using \eqref{eq:Kdef}.

\textit{Third descent: row 5.}
The remaining case is $\alpha\ge3$, $\beta\ge4$. For $k\ge4$,
\eqref{eq:Kstage} forces $X\equiv0\pmod p$. Hence $x=p^2X$, $y=p^2Y$,
and division by $p^4$ gives
\begin{equation}\label{eq:H2}
 Y^2\equiv p^2X^3+(A/p^2)X+B/p^4\pmod{p^{k-4}}.
\end{equation}
Lemma \ref{lem:descent} with $(a,b,c)=(2,2,4)$ gives the multiplicity
$p^4$, so $M_4=p^4$.
If $\beta=4$, its reduction is $Y^2=B/p^4\ne0$. Lemma \ref{lem:hensel} gives
row 5 exactly as in row 3.

\textit{Fourth descent: rows 6--8.}
Now assume $\alpha\ge3$, $\beta\ge5$. For $k\ge5$, \eqref{eq:H2} forces
$Y\equiv0\pmod p$, so the original coordinates satisfy
$x=p^2X$, $y=p^3Y$. Dividing the original equation by $p^5$ gives
\begin{equation}\label{eq:finalstage}
 pY^2\equiv pX^3+(A/p^3)X+B/p^5\pmod{p^{k-5}}.
\end{equation}
Its multiplicity is $p^5$, by Lemma \ref{lem:descent} with
$(a,b,c)=(2,3,5)$, and $M_5=p^5$.
If $\alpha=3$, the derivative in $X$ is a unit and the reduced equation
has $p$ solutions, proving row 6. If $\alpha\ge4$, $\beta=5$, the reduced
equation is a nonzero constant equal to zero, proving row 7 for $k\ge6$.
Finally, if $\alpha\ge4$, $\beta\ge6$, every coefficient in
\eqref{eq:finalstage} is divisible by $p$. For $k\ge6$ division by $p$
leaves the original Weierstra\ss{} equation with parameters $A/p^4,B/p^6$
and modulus $p^{k-6}$. Altogether this is Lemma \ref{lem:descent} with
$(a,b,c)=(2,3,6)$, of multiplicity $p^7$. Therefore
\[
 M_k=p^7N_{k-6}(A/p^4,B/p^6),
\]
including $k=6$ by $N_0=1$. This proves all rows and their stated initial
values.\hfill$\square$

\subsection{Closing the repeated-root cubic branch}\label{sec:ramified}
\begin{theorem}[ramified cubic count]\label{thm:K}
Suppose $p\nmid a$ or $p\nmid b$, and put
\[
 D=4a^3+27b^2,\qquad R=\#\{x\in\Fp:x^3+ax+b=0\}.
\]
For $\ell\ge1$, if $p\nmid D$, then
\begin{equation}\label{eq:Ksimple}
 K_\ell(a,b)=R p^\ell.
\end{equation}
If $p\mid D$, then $p\nmid ab$. Put $m=\vp(D)\ge1$, allowing $m=\infty$.
Then
\begin{equation}\label{eq:Kdouble}
 K_\ell(a,b)=
 \begin{cases}
  2p^\ell,&\ell\le m,\\
  (2+\varepsilon)p^\ell,&\ell>m,
 \end{cases}
\end{equation}
where, only in the second case, set $D_0=D/p^m$ and
\begin{equation}\label{eq:eps}
 \varepsilon=\chi\!\left(-D_0(2ab)^{m\bmod2}\right)\in\{-1,1\}.
\end{equation}
If $D=0$, only the first case is used. As before, $K_0=1$.
\end{theorem}

Above the double root, a change of variable reduces the equation
to the ramified quadratic form counted in Lemma~\ref{lem:J}.

\begin{lemma}[removing the higher cubic term]\label{lem:sqrtmap}
For a unit $d\in\Z_p^\times$, the map
\[
 z\longmapsto z\sqrt{1+pz/d},
\]
where the square root is the one congruent to $1$ modulo $p$, induces a
bijection modulo $p^n$ for every $n\ge1$. It transforms
$dz^2+pz^3$ into $dw^2$.
\end{lemma}
\begin{proof}
The square root exists uniquely by Corollary \ref{cor:hensel1}(a). It depends only on
$z$ modulo $p^n$. Write it as $s(z)$. For $z,z'\in\Z_p$,
\[
 s(z)-s(z')=\frac{(p/d)(z-z')}{s(z)+s(z')},
\]
and therefore
\[
 zs(z)-z's(z')=(z-z')\left(s(z)+\frac{pz'/d}{s(z)+s(z')}\right).
\]
The expression in parentheses is a unit, congruent to $1$ modulo $p$.
The map thus preserves valuations of differences and is injective, hence
bijective, on each finite residue ring. Its square gives the claimed
polynomial identity.
\end{proof}

\begin{lemma}[ramified quadratic count]\label{lem:J}
For $d\in\Z_p^\times$, define
\[
 J_n(c)=\#\{(y,z)\bmod p^n:y^2-pdz^2\equiv c\pmod{p^n}\},\qquad J_0=1.
\]
For $n\ge1$, put $\nu=\vp(c)$. If $\nu\ge n$, then $J_n(c)=p^n$.
If $\nu<n$ and $c_0=c/p^\nu$, then
\[
 J_n(c)=
 \begin{cases}
  (1+\chi(c_0))p^n,&\nu\text{ even},\\
  (1+\chi(-dc_0))p^n,&\nu\text{ odd}.
 \end{cases}
\]
\end{lemma}
\begin{proof}
By Lemma \ref{lem:normform}(c),
$\vp(y^2-pdz^2)=\min(2\vp(y),\,1+2\vp(z))$.
For $\nu\ge n$, the congruence therefore says exactly
$p^{\lceil n/2\rceil}\mid y$ and $p^{\lfloor n/2\rfloor}\mid z$,
giving $p^n$ pairs.

If $\nu=2t<n$, write $y=p^t y_1$, $z=p^t z_1$ and divide by $p^{2t}$.
The reduced right-hand side is a unit; modulo $p$ there are
$p(1+\chi(c_0))$ smooth points. Lemma \ref{lem:hensel} to modulus $p^{n-2t}$,
and the $p^{2t}$ unused higher-digit choices (Lemma \ref{lem:descent}), give
$(1+\chi(c_0))p^n$.
If $\nu=2t+1<n$, then 
$y=p^{t+1}y_1$, $z=p^t z_1$. Division gives
$p y_1^2-dz_1^2=c_0$ modulo $p^{n-2t-1}$. It has
$p(1+\chi(-c_0/d))$ smooth points modulo $p$, and $2t+1$ unused higher
digits. Since $\chi(-c_0/d)=\chi(-dc_0)$, the stated answer follows.
\end{proof}

\begin{proof}[Proof of Theorem~\ref{thm:K}]
If $D\not\equiv0\pmod p$, every root of $g(X)=X^3+aX+b$ modulo $p$ is
simple. Each root, together with arbitrary $Y\bmod p$, gives $p$ smooth
points on $g(X)-pY^2=0$. Lemma~\ref{lem:hensel} proves \eqref{eq:Ksimple}.

Suppose $p\mid D$. By Lemma \ref{lem:cubic}, there is one double root
$r\ne0$ and one simple root $-2r$, both in $\Fp$. By Corollary
\ref{cor:hensel1}(b), lift $r$
uniquely to a root $\rho\in\Z_p^\times$ of $3X^2+a=0$. Set
\[
 d=3\rho,\qquad c=g(\rho)=b-2\rho^3.
\]
Then, by Lemma \ref{lem:shift},
\begin{equation}\label{eq:critical}
 g(\rho+u)=c+du^2+u^3,\qquad
 D=27c(4\rho^3+c).
\end{equation}
The second factor and $27$ are units; hence $\vp(c)=\vp(D)=m$.

The simple root contributes $p^\ell$. Above the double root write
$X=\rho+pz$. The equation becomes
\[
 pY^2=c+p^2dz^2+p^3z^3.
\]
After division by $p$, Lemma~\ref{lem:sqrtmap} removes the cubic term.
There is one unused higher digit of $Y$, so this branch contributes
\[
 pJ_{\ell-1}(c/p).
\]
This also holds at $\ell=1$ using $J_0=1$.
If $\ell\le m$, Lemma~\ref{lem:J} gives $p^\ell$, so $K_\ell=2p^\ell$.

Suppose $\ell>m$, and put $c_0=c/p^m$. Its quadratic character can be
read off without computing $\rho$. Equation~\eqref{eq:critical} implies
\[
 \chi(c_0)=\chi(d)\chi(D_0),
\]
because $27\cdot4\rho^3/(3\rho)=36\rho^2$ is a nonzero square modulo $p$.
Moreover $a\equiv-3\rho^2$, $b\equiv2\rho^3$ gives, as in Lemma \ref{lem:cubic}(iv),
\[
 \chi(d)=\chi(-2ab).
\]
The valuation of $c/p$ is $m-1$. If $m$ is odd,
Lemma~\ref{lem:J} uses $\chi(c_0)=\chi(-2abD_0)$; if $m$ is even,
it uses $\chi(-dc_0)=\chi(-D_0)$. These are exactly
\eqref{eq:eps}. Adding the simple-root contribution proves
\eqref{eq:Kdouble}.
\end{proof}

\subsection{Eliminating all repeated scaling}\label{sec:compression}
\begin{theorem}[nonrecursive compression]\label{thm:compression}
For $k\ge1$ and arbitrary $A,B\in\Z$, put
\[
 t=\min\left(\left\lfloor\frac{\vp(A)}4\right\rfloor,
              \left\lfloor\frac{\vp(B)}6\right\rfloor,
              \left\lfloor\frac{k}6\right\rfloor\right),
 \quad \ell=k-6t,\quad A_t=A/p^{4t},\quad B_t=B/p^{6t}.
\]
The minimum uses the convention $\lfloor\infty\rfloor=\infty$; in particular
$t$ is finite. Then
\begin{equation}\label{eq:compression}
 {N_k(A,B)=p^{k-1}(p^t-1)+p^{7t}N_\ell(A_t,B_t).}
\end{equation}
The count on the right is evaluated without another scaling step:
if $\ell=0$ it is $1$; if $\ell<6$ it is a base value or another terminal
case; otherwise $p^4\nmid A_t$ or $p^6\nmid B_t$, so row 8 cannot apply.
The remaining cases are good reduction, a node, or rows 1--7 of
Theorem~\ref{thm:cusp}.
\end{theorem}
\begin{proof}
Whenever a scaling step is permitted, \eqref{eq:split1} and row 8 give
\[
 N_k(A,B)=(p-1)p^{k-1}+p^7N_{k-6}(A/p^4,B/p^6).
\]
After $t$ such steps the accumulated smooth contribution is
\[
 \sum_{i=0}^{t-1}p^{7i}(p-1)p^{k-6i-1}
 =(p-1)p^{k-1}\sum_{i=0}^{t-1}p^i
 =p^{k-1}(p^t-1).
\]
The remaining contribution is $p^{7t}N_\ell(A_t,B_t)$.
The definition of $t$ gives the asserted stopping condition.
The same calculation is valid for $t=0$, when the sum is empty.
\end{proof}

\begin{corollary}[the identically cuspidal equation]
For $A=B=0$ and $k\ge1$, writing $u=\lfloor(k-1)/6\rfloor+1$, one has
\[
 N_k(0,0)=p^{k-1}(p^u-1)+p^{k-\lceil k/3\rceil+\lfloor k/2\rfloor}.
\]
\end{corollary}
\begin{proof}
For $\vp(x)=2j$ with $6j<k$, half the leading units of $x$ are quadratic
residues. Their contribution to the square-root count of $x^3$ is
$(p-1)p^{k+j-1}$. Odd $\vp(x)$ with $3\vp(x)<k$ contributes nothing.
Finally, $p^{\lceil k/3\rceil}\mid x$ makes $x^3\equiv0\pmod{p^k}$ and
contributes $p^{k-\lceil k/3\rceil+\lfloor k/2\rfloor}$.
Summing the first geometric progression proves the formula. It also
independently checks arbitrarily many six-step cycles.
\end{proof}

\subsection{Arithmetic terms with variable $k$}\label{sec:terms}

Theorem~\ref{thm:compression} removes the repeated scaling from the
count. To express the resulting finite case distinction by an
arithmetic term, we also need a term for the number of roots of
$X^3+aX+b$ modulo $p$. We obtain it by the coefficient extraction used
in Lemma~\ref{lem:extract}.

\begin{proposition}[a term for the residual cubic root count]\label{prop:R}
Let $\bar a=a\bmod p$, $\bar b=b\bmod p$, and put
\[
 Q=2^{p^2},\qquad H=(Q^3+\bar aQ+\bar b)^{p-1}.
\]
Then, for every prime $p\ge5$,
\begin{equation}\label{eq:Rterm}
 R=\left(1+
 \left(\left\lfloor H/Q^{p-1}\right\rfloor\bmod Q\right)+
 \left(\left\lfloor H/Q^{2(p-1)}\right\rfloor\bmod Q\right)
 \right)\bmod p.
\end{equation}
\end{proposition}
\begin{proof}
Write $(X^3+\bar aX+\bar b)^{p-1}=\sum_j c_jX^j$.
The coefficients are nonnegative and
\[
 \sum_j c_j=(1+\bar a+\bar b)^{p-1}\le(2p-1)^{p-1}<2^{p^2}=Q.
\]
The last inequality is Lemma \ref{lem:inequality}. Thus the two digit extractions in \eqref{eq:Rterm}
return $c_{p-1}$ and $c_{2(p-1)}$ exactly.

For each $x\in\Fp$, $1-g(x)^{p-1}$ is the indicator that $g(x)=0$.
Summing and using the finite-field power sums gives
\[
 R\equiv c_{p-1}+c_{2(p-1)}+c_{3(p-1)}\pmod p.
\]
The highest coefficient is $1$, and $0\le R\le3<p$, proving equality
with the least nonnegative residue in \eqref{eq:Rterm}.
\end{proof}

The valuation $\vp$ with $p$ a variable is the arithmetic term of
Proposition \ref{prop:vp}, and the capped valuation is \eqref{eq:cap}.

For natural numbers $x,y$, write
\[
 E(x,y)=1\tsub\bigl((x\tsub y)+(y\tsub x)\bigr).
\]
This is the equality indicator. For a unit $u\bmod p$, put
\[
 e(u)=((u\bmod p)^{(p-1)/2})\bmod p,\qquad s(u)=E(e(u),1).
\]
Then $1+\chi(u)=2s(u)$ and $2+\chi(u)=1+2s(u)$.
A negative residue $-u\bmod p$ is represented, with natural intermediates,
by $(p\tsub(u\bmod p))\bmod p$. Thus the displayed Legendre symbols do not
require a negative intermediate arithmetic value.

\begin{corollary}[fixed finite expression]\label{cor:term}
For $A,B\ge0$, prime $p\ge17$ and variable $k\ge1$, the cusp count is a
fixed arithmetic term, using the good-reduction and nodal formulas
for the residual count in Theorem~\ref{thm:compression}.
\end{corollary}
\begin{proof}
Use \eqref{eq:cap} with $L=k$ to compute the two valuations used in $t$;
capping does not change their minimum with $\lfloor k/6\rfloor$.
Theorem~\ref{thm:compression} makes the residual coefficients by exact
divisions and reduces the problem to a fixed case table. For a residual
cusp, use Theorems~\ref{thm:cusp} and \ref{thm:K}. The only cubic root
count is the arithmetic term \eqref{eq:Rterm}.
In \eqref{eq:Kdouble}, only $\min(m,\ell)$ is needed. The unit $D/p^m$
is used only if $m<\ell$, so this cap gives the exact required valuation
and the unit's residue. All valuations therefore use positive inputs
through \eqref{eq:cap}; there is no term $\vp(0)$.

Every condition in the finite table is a divisibility, equality,
inequality or parity test, expressed by the indicators above.
Multiply each branch expression by the indicator of its case and add
the resulting finitely many terms. On inputs outside a branch, any
negative exponent is replaced by zero using truncated subtraction.
The denominators remain positive: they are powers of $p$ or the
positive quantities in Proposition~\ref{prop:vp}. The good-reduction
and nodal branches use the terms already established. Thus the number
of arithmetic operations is independent of $k$.

\end{proof}

For $p=5,7,11,13$, the same construction applies if the initial count
$N_1$ is supplied by Theorem~\ref{thm:composite} or by four fixed tables
of point counts indexed by the coefficients modulo $p$.

\begin{remark}

If the factorization of $n$ is known and $\gcd(n,6)=1$, the Chinese
remainder theorem gives the count over $\Z/n\Z$ as a product of the
counts over its prime-power factors. Each factor is supplied by the
local formulas above. The all-moduli term of
Theorem~\ref{thm:composite} can be evaluated without a factorization.

\end{remark}

\section{Relation to Tate's algorithm}\label{sec:tate}
For a generically nonsingular minimal short Weierstra\ss{} equation at
$p\ge5$, the exits in the cusp table correspond to
\[
 II,\ III,\ IV,\ I_0^*\text{ or }I_m^*,\ IV^*,\ III^*,\ II^*.
\]
In row 4, $p\nmid D$ gives $I_0^*$, while $m=\vp(D)\ge1$ gives $I_m^*$;
the original elliptic discriminant has valuation $m+6$. Row 8 is the
nonminimal scaling. These correspondences agree with the stepwise
analysis in Cremona--Sadek \cite[Sections 2 and 5.1]{CS}.

The count modulo $p^k$ depends on the given integral equation.
A nonminimal equation can have cuspidal reduction even when a minimal
equation for the same elliptic curve has good or multiplicative
reduction. The scaling contribution in
Theorem~\ref{thm:compression} must therefore be retained.
When $\Delta=0$, the counting formulas still apply, but the generic
cubic is singular and has no elliptic Kodaira symbol.

\section{Numerical verification}\label{sec:numerics}

The point-counting formulas were checked against independent computations
using exact integer arithmetic. All reported comparisons agreed.

\paragraph{Arbitrary moduli.}
For each $n=1,\ldots,48$, six coefficient pairs were tested in
Theorem~\ref{thm:composite}, including coefficients at least $n$ and
cubics singular modulo $n$. In each of the $288$ comparisons, the point
count was obtained by direct enumeration, and the packed integer $M$
was constructed both from the individual values $\delta(E)$ and from
the $15$ products $S(i,j,k)$. The two constructions of $M$ agreed as
integers, and the value of the term agreed with the direct count.
At $n=60$, the packed integer had approximately $1.9\cdot10^9$ binary
digits and its evaluation took $35$ seconds. At $n=100$, it had
approximately $2.26\cdot10^{10}$ binary digits; two evaluations each
took about $8$ minutes of wall time on one processor core. These
evaluations also agreed with direct counts.

\paragraph{Prime fields.}
The complete term of Theorem~\ref{thm:ztopz}, with digit width $p^2$,
was evaluated for all primes $17\le p\le400$ and six coefficient pairs
per prime. The $432$ values were compared with direct point counts.
The congruence and recentering step were checked separately for all
primes $17\le p\le10^6$, with three coefficient pairs per prime,
giving $235{,}476$ comparisons. This larger computation used the
character sum to obtain the residue $r$; it did not evaluate the
packed integer. The prime-field results were reproduced by a second
implementation. Complete evaluations at $p=1009$ and $p=2003$ took
$5.3$ and $72$ seconds of wall time, respectively, on one processor
core.

\paragraph{Extension fields.}
Theorem~\ref{thm:ext} was compared with direct counts in fields
$\mathbb F_{p^2}$ and $\mathbb F_{p^3}$ constructed for
$p\in\{17,19,23\}$. Five coefficient pairs for each prime, including a
node and a cusp, gave $30$ comparisons. The complete term of
Theorem~\ref{thm:term} was also evaluated for
$p\in\{17,19,23,31,101\}$ and $1\le k\le12$. Four coefficient pairs
were used for each prime, with the singular pair at $p=31$ excluded.
The resulting $228$ values were compared with the recurrence in
Theorem~\ref{thm:ext}.

\paragraph{Good reduction and nodes modulo $p^k$.}
For these checks, the independent count was obtained by tabulating the
number of square roots of each residue modulo $p^k$, then summing the
entry at $x^3+Ax+B$ over $x$. This takes $O(p^k)$ operations.
\begin{itemize}
\item For good reduction, the primes were $5,7,11,13,17,19,23$.
  The seven coefficient pairs for each prime were
  $(1,1)$, $(2,3)$, $(0,1)$, $(1,0)$, $(5,7)$, $(3,11)$ and
  $(p+2,p^2+5)$. Retaining $p\nmid\Delta$, $2\le k\le4$ and
  $p^k\le200{,}000$ gave $128$ comparisons with
  Theorem~\ref{thm:good}.
\item For nodes, the systematic families were
  $A=-3\rho^2$, $B=2\rho^3+p^m$.
  The primes were $5,7,11,13,17,19,23,29$, with
  $1\le\rho\le\min(8,p-1)$ and $1\le m\le7$.
  The families $B=2\rho^3$ were also included. Restricting to
  $1\le k\le7$ and $p^k\le300{,}000$ gave $2048$ comparisons
  with Theorem~\ref{thm:node}.
\item A further $592$ comparisons used random coefficients with good
  or nodal reduction, for primes $5\le p\le23$, $1\le k\le7$ and
  $p^k\le300{,}000$. The nodal invariants were computed from the
  coefficients as
  $v_p((\Delta\bmod p^k)+p^k)$ and $\legendre{-2AB}$.
\item The valuation term in Proposition~\ref{prop:vp} was checked
  for $10$ primes and all inputs $1\le n\le199$.
\end{itemize}

\paragraph{Cusps modulo $p^k$.}
Table~\ref{tab:cusp-checks} records $7999$ comparisons of the cusp
formulas and their auxiliary terms. The first five families, totaling
$5233$ comparisons, used the independent square-root count described
above. The largest modulus was $5^9=1{,}953{,}125$.
The $480$ checks for $A=B=0$ compared the compressed formula with the
separate valuation sum in Section~\ref{sec:compression}, allowing
exponents up to $k=120$ without enumerating residues modulo $p^k$.

\begin{table}[htbp]
\centering
\small
\renewcommand{\arraystretch}{1.25}
\begin{tabular}{@{}p{11.0cm}r@{}}
\toprule
Check family & Comparisons\\
\midrule
Cusp coefficients $A=pa,B=pb$, all $0\le a,b<p$,
 $p=5,7,11$, $1\le k\le5$ &975\\
Cusp valuation grid: $1\le\alpha\le7$, $1\le\beta\le9$,
 three leading-unit pairs; $p=5,7,11$, respectively $k\le7,6,5$ &3402\\
Cusp repeated-root families, both valuation parities and exact zero discriminant &588\\
Nonminimal families, including scaled good, nodal and cuspidal equations,
 $p=5$, $6\le k\le9$ &28\\
Random cusp coefficients for $p=5,7,11,17$, seed $20260910$ &240\\
Auxiliary ramified cubic: coefficient grids &768\\
Auxiliary ramified cubic: families with prescribed discriminant valuations &504\\
Cubic-root arithmetic term, all coefficient pairs modulo $p$ for
 $p=5,7,11,13,17,19$ &1014\\
Pure cusp $A=B=0$, independent valuation sum,
 $p=5,7,11,17$, $1\le k\le120$ &480\\
\midrule
Total &7999\\
\bottomrule
\end{tabular}

\caption{Exact integer comparisons for the cuspidal formulas over
$\Z/p^k\Z$ and the auxiliary cubic counts. Here $\alpha=v_p(A)$ and
$\beta=v_p(B)$. Each comparison tests a formula against an independent
count or valuation sum; all $7999$ comparisons agreed.}
\label{tab:cusp-checks}
\end{table} 

\section*{Acknowledgments}

This paper was produced in an AI-assisted workflow. AI assistants helped
explore candidate constructions, draft proofs, and write software for numerical experiments.

\end{document}